\documentclass[final]{amsart}
\usepackage{amsmath}
\usepackage{amssymb}
\usepackage{amsthm}
\usepackage{amscd}
\usepackage{mathtools}
\usepackage{marginnote}
\usepackage{todonotes}
\usepackage{bbm}
\usepackage{bm}
\usepackage[notcite,notref]{showkeys}
\usepackage[utf8]{inputenc}
\usepackage{tikz-cd}
\usepackage{cite}
\usepackage[colorlinks=true,linkcolor=black,citecolor=black,urlcolor=blue]{hyperref} 
\usepackage{cleveref}
\newtheorem{theorem}{Theorem}[section]
\newtheorem*{theorem*}{Theorem}

\newtheorem*{conjecture*}{Conjecture}

\newtheorem*{question*}{Question}

\newtheorem*{guess*}{Guess}

\newtheorem*{problem*}{Problem}

\newtheorem{lemma}[theorem]{Lemma}
\newtheorem*{lemma*}{Lemma}

\newtheorem*{exercise*}{Exercise}
\newtheorem{proposition}[theorem]{Proposition}
\newtheorem*{proposition*}{Proposition}

\newtheorem{corollary}[theorem]{Corollary}
\newtheorem*{corollary*}{Corollary}

\theoremstyle{definition}

\newtheorem{definition}[theorem]{Definition}
\newtheorem*{definition*}{Definition}
\newtheorem{remark}[theorem]{Remark}

\newtheorem*{example*}{Example}
\newtheorem*{examples*}{Examples}

\newcommand{\up}[1]{^{(#1)}}
\newcommand{\pres}[1]{\left\langle #1 \right\rangle}

\renewcommand{\bar}{\overline}
\usepackage{mathtools}

\renewcommand{\AA}{\mathbb{A}}

\newcommand{\CC}{\mathbb{C}}

\newcommand{\FF}{\mathbb{F}}
\newcommand{\GG}{\mathbb{G}}

\newcommand{\NN}{\mathbb{N}}

\newcommand{\QQ}{\mathbb{Q}}

\newcommand{\TT}{\mathbb{T}}

\newcommand{\ZZ}{\mathbb{Z}}

\newcommand{\Cc}{\mathcal{C}}

\newcommand{\Ec}{\mathcal{E}}

\newcommand{\Oc}{\mathcal{O}}
\newcommand{\Pc}{\mathcal{P}}

\newcommand{\Sc}{\mathcal{S}}

\newcommand{\Zc}{\mathcal{Z}}

\newcommand{\mf}{\mathfrak{m}}

\newcommand{\Xf}{\mathfrak{X}}

\newcommand{\rarrow}{\rightarrow}

\newcommand{\onto}{\twoheadrightarrow}
\newcommand{\into}{\hookrightarrow}
\newcommand{\isomto}{\xrightarrow{\sim}}

\newcommand{\normal}{\lhd}

\newcommand{\rhobar}{\bar{\rho}}
\newcommand{\w}{\wedge}

\newcommand{\cyc}{\operatorname{cyc}}
\newcommand{\red}{\operatorname{red}}

\newcommand{\End}{\operatorname{End}}
\newcommand{\Hom}{\operatorname{Hom}}

\newcommand{\rec}{\operatorname{rec}}
\newcommand{\Gal}{\operatorname{Gal}}

\newcommand{\Ind}{\operatorname{Ind}}

\newcommand{\Spec}{\operatorname{Spec}}
\newcommand{\Spf}{\operatorname{Spf}}

\newcommand{\rank}{\operatorname{rank}}

\newcommand{\ch}{\operatorname{char}}

\newcommand{\ad}{\operatorname{ad}}

\newcommand{\diag}{\operatorname{diag}}

\newcommand{\id}{\mathrm{id}}

\renewcommand{\ch}{\mathrm{ch}}
\newcommand{\FS}{\mathcal{FS}}
\newcommand{\sbar}{\bar{s}}

\newcommand{\BqG}{B_{q,\hat{G}}}
\newcommand{\BqL}{B_{q,\hat{L}}}
\newcommand{\Ghat}{\hat{G}}
\newcommand{\Lhat}{\hat{L}}
\newcommand{\rk}{\operatorname{rk}}
\newtheorem*{theoremA}{Theorem A}
\newtheorem*{theoremB}{Theorem B}

\title{Generic local deformation rings when $l \neq p$.}    \author{Jack
    Shotton}

\begin{document}
\maketitle
\begin{abstract}We determine the local deformation rings of sufficiently generic mod $l$ representations of the Galois
  group of a $p$-adic field, when $l \neq p$, relating them to the space of $q$-power-stable semisimple conjugacy
  classes in the dual group.  As a consequence we give a local proof of the $l \neq p$ Breuil--M\'{e}zard conjecture of
  the author, in the tame case.
  \end{abstract}


\section{Introduction}
\label{sec:introduction}

We study the moduli space $\Xf$ of $n$-dimensional $l$-adic representations of the tame Weil group of a $p$-adic field
$F$, when $l \neq p$ are primes and $n \geq 1$ is an integer.  The main geometric result, Theorem~\ref{thm:def-rings},
is a simple description of the completion of $\Xf$ at a sufficiently general point of its special fibre.  We then apply
this to give a purely local proof of the author's $l \neq p$ analogue of the Breuil--M\'{e}zard conjecture in the tame
case --- see Theorem~\ref{thm:tame-BM}.  This was formulated, and proved for $l \geq 2$ by global automorphic methods,
in \cite{shotton-gln}.  This result links congruences between representations of $GL_n(k)$, where $k$ is the residue
field of $F$, and `congruences' between irreducible components of $\Xf$; for more background and motivation, see the
introduction to~\cite{shotton-gln}.

We give a more precise description of our results and methods in the most critical case.  Let $W_t$ be the tame Weil
group and $I_t$ be the tame inertia group of $F$, and let $(\Oc, E, \FF)$ be a sufficiently large $l$-adic coefficient
system.  Let $q$ be the order of $k$, the residue field of $F$.  Suppose that $\rhobar : W_t \rarrow GL_n(\FF)$ is a
representation such that $\rhobar(g)$ is regular unipotent for any topological generator $g$ of $I_t$.

Let $\hat{T}$ be a maximal split torus in $GL_{n, \Oc}$ and let $W$ be the Weyl group.
We have a `characteristic polynomial' map
\[\ch : GL_{n, \Oc} \rarrow \hat{T}/W.\] We consider the $q$-fixed
subscheme of $\hat{T}/W$, which we denote by
\[(\hat{T}/W)^q,\]
and its localisation at the point $\bar{e}$ of its special fibre corresponding to the identity in $\hat{T}(\FF)$.
\begin{theoremA}[Theorem~\ref{thm:regular-unipotent}] The morphism 
  \[\Xf^\wedge_{\rhobar} \rarrow (\hat{T}/W)^q_{\bar{e}}\]
  defined by $\rho \mapsto \ch(\rho(\sigma))$ is formally smooth, where $\Xf^\wedge_{\rhobar}$ is the completion of
  $\Xf$ at the point corresponding to $\rhobar$.
\end{theoremA}
Note that the completion $\Xf^\wedge_{\rhobar}$ is simply the framed deformation ring of $\rhobar$.  The proof of
Theorem~A is an elaboration of the proof of Proposition~7.10 in \cite{shotton-gln}.

More generally, to each irreducible component $\mathcal{C}$ of the special fibre of $\Xf$ we associate a Levi subgroup
$\hat{M} \subset GL_{n, \Oc}$ containing $\hat{T}$, with Weyl group $W_{\hat{M}} \subset W$, and an $\FF$-point
$\bar{s}$ of $(\hat{T}/W_{\hat{M}})^q$.  For sufficiently general points $\rhobar$ on $C$ we construct a morphism
\[\Xf^\wedge_{\rhobar} \rarrow (\hat{T}/W_{\hat{M}})^q_{\bar{s}}\]
and show that it is formally smooth.  See Theorem~\ref{thm:def-rings}.  The proof proceeds by reducing first to the case
that $\rhobar(g)$ is unipotent for all $g\in I_t$ (see Section~\ref{sec:unipotent-reduction}), and then to the situation
of Theorem~A (see Corollary~\ref{cor:reduce-to-M}).

    We explain the application to the $l \neq p$ ``Breuil--M\'{e}zard conjecture'' of \cite{shotton-gln} in the `tame case',
    whose statement we briefly recall.  Set $G = GL_{n, k}$.  Let $\Zc(\Xf)$ (resp. $\Zc(\Xf_\FF)$) be the free abelian group on
    the irreducible components of $\Xf$ (resp. $\Xf_{\FF}$).  Let $K_E(G(k))$ (resp. $K_\FF(G(k))$) be the Grothendieck
    groups of representation of $G(k)$ over $E$ (resp. $\FF$).  There is a `cycle map'
  \[\cyc : K_E(G(k)) \rarrow \Zc(\Xf)\]
  (see Section~\ref{sec:combining}) motivated by the local Langlands correspondence, and natural `reduction maps'
  $\red : K_E(G(k)) \rarrow K_\FF(G(k))$ and $\red : \Zc(\Xf) \rarrow \Zc(\Xf_{\FF})$.  We then have
  \begin{theoremB}[Theorem~\ref{thm:tame-BM}] There is a unique map $\bar{\cyc} : K_\FF(G(k)) \rarrow \Zc(\Xf_\FF)$ such that the diagram
\[\begin{CD} K_E(G(k)) @>{\cyc}>> \Zc(\Xf) \\
    @V{\red}VV  @V{\red}VV \\
    K_{\FF}(G(k)) @>{\bar{\cyc}}>> \Zc(\Xf_{\FF})
  \end{CD} \]
commutes.
\end{theoremB}

It is enough to prove Theorem~B after formally completing at some $\rhobar$ on each component.  We explain how to do
this for $\rhobar$ as in Theorem~A.  Let $\Gamma$ be the (integral) Gelfand--Graev representation of $G(k)$ over
$\Oc$ --- it is a \emph{projective} $\Oc[G(k)]$ representation.  Let $B_{q,n}$ be the coordinate ring of
$(\hat{T}/W)^q$.  Via the `Curtis homomorphisms' we define a
homomorphism
\begin{equation}\label{eq:curtis}B_{q,n} \rarrow \End(\Gamma)\otimes E\end{equation}
which restricts to a homomorphism
\[B_{q,n,\bar{e}} \rarrow \End(e\Gamma)\otimes E\] for a certain idempotent $e \in \Oc[G(k)]$.  (For this, we need a
result of Brou\'{e}--Michel in \cite{MR983059} on the blocks of $\Oc$-representations of $G(k)$).  The special fibre of
$\Xf_{\rhobar}^\wedge$ has a unique irreducible component $\Cc$ and we may define
\[\bar{\cyc}(\sigma) = \dim_{\FF} \Hom(\Gamma, \sigma)[\Cc].\]
That this works is essentially a consequence of the projectivity of $\Gamma$, together with Theorem~A.

The proof of Theorem~B is carried out in Sections~\ref{sec:rep-gln} and~\ref{sec:combining} --- in
Section~\ref{sec:rep-gln} we recall the necessary material on Gelfand--Graev and Deligne--Lusztig representations, and
this is applied to Theorem~B in Section~\ref{sec:combining}.

The functor
$\Hom(\Gamma, \cdot)$ plays the role in this proof that the functor $M_\infty(\cdot)$ plays in the global proof via
patching, and so one could see the relationship between this article and \cite{shotton-gln} as being parallel to that
between \cite{MR3306557} and \cite{Kisin2009-FontaineMazur}.

Helm and Moss have proved in \cite{helm2016curtis} and \cite{1610.03277} that the local Langlands correspondence in
families, conjectured in \cite{1104.0321}, exists.  As a consequence, or byproduct, of their proof, it follows that the
map (\ref{eq:curtis}) actually defines an isomorphism
\[B_{q,n} \isomto \End(\Gamma).\] This is a result purely in the representation theory of finite groups, and it would be
interesting to have an elementary proof.  For general connected reductive groups, results on the endomorphism rings of
integral Gelfand--Graev representations were obtained by Bonnaf\'{e} and Kessar in \cite{MR2441999}, under the
assumption that $l$ does not divide the order of the Weyl group (and is distinct from $p$).

The idea of using the Gelfand--Graev representations came from \cite{helm2016curtis}.  Having proved Theorem~A, I asked
David Helm whether the map~(\ref{eq:curtis}) could be an isomorphism and our correspondence turned up an error in an
earlier version of \cite{helm2016curtis}, which was corrected by him using, among other things, the
map~(\ref{eq:curtis}) and the idea behind the proof of Theorem~A.  He was then able to show that the
map~(\ref{eq:curtis}) was indeed an isomorphism, as a consequence of his work with Moss.  There are other ways to deduce
Theorem~B from Theorem~A; my original method was a complicated combinatorial induction.

We take some care to write things in a way that is independent of a choice of topological generator of $I_t$. Thus
instead of $(\hat{T}/W)^q$ we actually use the space of $q$-stable $W$-orbits of homomorphisms $I_t \rarrow \hat{T}$.
Points of this space over $E$ then canonically parametrise Deligne--Lusztig representations of $GL_n(k)$ over $E$, a
construction we learned from \cite{MR2480618}.


\subsection{Acknowledgments}

Parts of this work were conducted at the University of Chicago and at the Max Planck Institute for Mathematics, and I
am grateful both institutions for their support.  I thank Andrea Dotto and David Helm for helpful conversations and
correspondence.

\subsection{Notation}
\label{sec:notation}

An $l$-adic coefficient system is a triple $(E, \Oc, \FF)$ where: $E$ is a finite extension of $\QQ_l$, $\Oc$ is its
ring of integers, and $\FF$ is its residue field.  We then define $\Cc_{\Oc}$ to be the category of artinian local
$\Oc$-algebras with residue field $\FF$, and $\Cc^\wedge_{\Oc}$ be the category of complete artinian local
$\Oc$-algebras that are inverse limits of objects of $\Cc_{\Oc}$.  We also consider affine formal schemes of the form
$\Spf(R)$ for $R$ an object of $\Cc_{\Oc}$ or $\Cc_{\Oc}^\w$ (taken with respect to the $\mf_R$-adic topology); these
form categories which we denote by $\FS_{\Oc}$ or $\FS_{\Oc}^\w$ respectively (and which are canonically isomorphic to
the opposite categories of $\Cc_{\Oc}$ and $\Cc_{\Oc}^\w$).  For $X \in \FS_{\Oc}^\w$ and $A \in \Cc_{\Oc}^\w$ we write
$X(A) = \Hom_{\FS_{\Oc}^\w}(X, \Spf(A))$.  If $X/\Oc$ is a scheme locally of finite type, and $x \in X(\FF)$, then we
let $X^\w_x = \Spf\left(\varprojlim \Oc_{X, x}/\mf_{X,x}^n\right)$ be its formal completion, an object of
$\FS_{\Oc}^\w$.

If $A$ is a ring, we write $\diag(x_1, \ldots, x_n)$ for the diagonal matrix with entries $x_1, \ldots, x_n$.  If $\zeta
\in A$ and $n\in \NN$, then we write $J_n(\zeta)$ for the $n \times n$ Jordan block matrix with $\zeta$ on the diagonal
and $1$ on the superdiagonal.

\section{Moduli of Weil group representations}
\label{sec:parameters}

\subsection{Galois groups}
\label{sec:galois-groups}

Choose a maximal tamely ramified extension $F^{t}$ of $F$.  This induces an algebraic closure $\bar{k}$
of $k$.  For $n \in \NN$, let $k_n$ be the subextension of $\bar{k}/k$ having degree $n$ over $k$. Let
$G_t = \Gal(F^{t}/F)$.  The canonical homomorphism $G_t \rarrow G_k = \Gal(\bar{k}/k) \cong \hat{\ZZ}$ has kernel the tame
inertia subgroup $I_t$, and the Weil group $W_t \subset G_t$ is the preimage of $\ZZ$ under this homomorphism.

There is a canonical isomorphism
\[\omega : I_t \isomto \varprojlim k_n^\times\] where the inverse limit is under the norm maps $k_n \rarrow k_m$ for
$m \mid n$.  The exact sequence \[1 \rarrow I_t \rarrow G_t \rarrow G_k \rarrow 1\] splits, so that we have a canonical isomorphism
\[G_t \cong (\varprojlim k_n^\times) \rtimes G_k\] where $G_k$ acts on each $k_n^\times$ in the natural way.  More
concretely, if we choose a topological generator $\sigma \in I_t$ and lift $\phi \in G_t$ of arithmetic Frobenius, then
$G_t$ is isomorphic to the profinite completion of
\[\left\langle \phi, \sigma | \phi \sigma \phi^{-1} = \sigma^q\right\rangle.\] Note that, as a
topological group, this only depends on the integer $q$.  A pair $(\sigma, \phi)$
as above will be called (a choice of) \emph{standard (topological) generators} of $G_t$ (or $W_t$).

\subsection{Moduli spaces}
\label{sec:moduli}

Fix an $l$-adic coefficient system $(E, \Oc, \FF)$.  Let $\hat{G}$ be an algebraic group over $\Oc$ isomorphic to a product of
finitely many general linear groups (for the proofs of all the statements below, we can and do immediately reduce to the
case of $GL_n/\Oc$, but the slight extra generality will be useful later).

\begin{proposition}\label{prop:representable}
  The functor taking an $\Oc$-algebra $A$ to the set of continuous\footnote{We topologise any $\ZZ_l$-algebra $A$ as
    the direct limit of its finitely generated $\ZZ_l$-algebra, and give $\hat{G}(A)$ its canonical topology as the
    points of an affine scheme over a topological ring, as in \cite{conrad2012weil}.} homomorphisms
  \[\rho : W_t \rarrow \hat{G}(A)\]
  is representable by an affine scheme $\Xf^{\hat G}(q)$ of finite type over $\Oc$ that is reduced, $\Oc$-flat, and
  a local complete intersection of dimension $\dim_\Oc({\hat G}) + 1$.
\end{proposition}

\begin{remark}
    Forthcoming work of Dat, Helm, Kurinczuk, and Moss will show that the analogous result holds with
    $\hat{G}$ replaced by an arbitrary unramified connected reductive group over $\ZZ_l$.
\end{remark}

\begin{proof} We may and do assume that $\hat{G} = GL_n/\Oc$ for some $n$.  Choose standard topological generators
  $\sigma$ and $\phi$ of $W_t$, and let $W'_t$ be the subgroup they generate.  As $W'_t$ is finitely generated, it is
  clear that the functor taking $A$ to the set of homomorphisms $\rho : W'_t \rarrow \hat{G}(A)$ is representable by a
  finite-type affine scheme $\Xf$ over $\ZZ_l$.  Moreover, \cite[Proposition 6.2]{helm2016curtis} implies that $\Xf$
  enjoys the geometric properties that we are claiming for $\Xf^{\hat{G}}(q)$.
  \begin{lemma} Suppose that $A$ is a $\ZZ_l$-algebra and that $M$ is a finite $A$-module, free of rank $n$, with an
    $A$-linear action $\rho$ of $W'_t$.  Then there is a unique continuous $A$-linear action $\tilde{\rho}$ of $W_t$ on
    $M$ extending that of $W'_t$.
  \end{lemma}
  \begin{proof}
    Note first that every \emph{finite image} representation of $W'_t$ extends uniquely to a continuous representation
    of $W_t$.  Now suppose that $M$ is as in the lemma.  I claim that $(\sigma^{q^{n!} - 1} - 1)^n$ acts as zero on $A$.
    Indeed, it suffices to check that this holds for the universal representation of $W'_t$ over $\Xf$.  This in
    turn can be checked at geometric points in characteristic zero, since $\Xf$ is of finite type, $\ZZ_l$-flat,
    and reduced.  But at such points the eigenvalues of $\sigma$ are permuted by the $q$-power map, and so fixed by the
    $q^{n!}$-power map.  Thus they are all $(q^{n!}-1)$th roots of unity.  The result follows from the
    Cayley--Hamilton theorem.
 
    By the previous paragraph, the $\ZZ_l$-subalgebra $\Ec$ of $\End_A(M)$ generated by $\rho(\sigma)$ is a finitely
    generated $\ZZ_l$-module.  It follows that there is a finitely generated $\ZZ_l$-submodule $N$ of $M$ that generates
    $M$ as an $A$-module and that is preserved by $\sigma$, so that $\Ec \subset \End(N)$.  I claim that the map
    $k \mapsto \rho(\sigma)^k$ is a continuous map from $\ZZ$, equipped with the linear topology whose open ideals are
    $m\ZZ$ for $m$ coprime to $p$, to $\End(N)$.  If $k \equiv k' \mod q^{n!} - 1$, then by the previous paragraph
    $(\rho(\sigma)^{k - k'} - 1)^n = 0$.  It follows that, for every $s \in \NN$, there exists $r \in \NN$ such that
    $\rho(\sigma)^{k - k'} \equiv 1$ in $\End(N/l^sN)$ for all $k \equiv k' \mod (q^{n!} - 1)l^r$.  This is the
    required continuity.  We deduce that $\rho$ extends to a unique continuous map from the completion of
    $\langle \sigma \rangle$ with respect to this topology to $\Ec \subset \End(N)$.  This completion is canonically isomorphic to
    $I_t$, and we therefore obtain a continuous homomorphism $I_t \rarrow \Ec \subset \End(M)$.  It follows from the
    unicity that this extends to a continuous homomorphism $W'_t \rarrow \End(M)$.
\end{proof}
The proposition follows immediately, with $\Xf^{\hat G}(q) = \Xf$.
\end{proof}
\begin{remark}
The reason for formulating Proposition~\ref{prop:representable} with $W_t$ rather than the subgroup $W'_t$ used in the
proof is to get a moduli space whose definition does not require a choice of $\sigma$.
\end{remark}

\subsection{Parameters}
\label{sec:inertial}

Let $C$ be a field containing $\FF$ or $E$, and let ${\hat G}$ be as above.  In the following, we will usually omit the word
`tame', since that is the only case we consider in this article.

\begin{definition}
  A (tame) \emph{${\hat G}$-parameter over $C$} is a $\hat G(C)$-conjugacy class of homomorphisms $\rho : W_t \rarrow {\hat G}(C)$.

  A homomorphism $\tau : I_t \rarrow {\hat G}(C)$ is \emph{extendable} if it extends to a homomorphism $W_t
  \rarrow \hat{G}(C)$; equivalently, if $\tau$ is conjugate in $\hat{G}(C)$ to the homomorphism $\tau^q$. It is
  \emph{semisimple}/\emph{unipotent} if every element of its image is.

  A (tame) \emph{inertial ${\hat G}$-parameter over $C$} is a ${\hat G}(C)$-conjugacy class of extendable homomorphisms
  $\tau : I_t \rarrow {\hat G}(C)$.  It is \emph{semisimple}/\emph{unipotent} if every homomorphism in its conjugacy
  class is.  Since $I_t$ is pro-cyclic, any inertial ${\hat G}$-parameter has a unique Jordan decomposition
  $\tau = \tau_s \tau_u$ where $\tau_s$ is a semisimple inertial $\hat G$-parameter, $\tau_u$ is a unipotent inertial
  $\hat G$-parameter, and the images of $\tau_s$ and $\tau_u$ commute.
\end{definition}

For every inertial ${\hat G}$-parameter $\tau$ over $C$, let $\Xf^{\hat G}(q, \tau)$ be the Zariski closure of the $\bar{C}$-points
$\rho$ of $\Xf^{\hat G}(q)$ such that $\rho |_{I_t} \sim \tau$.  Then as in \cite[Proposition 2.6]{shotton-gln}, we have:

\begin{proposition} The assignment $\tau \mapsto \Xf^{\hat G}(q,\tau)$ is a bijection between semisimple inertial
  ${\hat G}$-parameters over $C$ and irreducible components of $\Xf^{\hat G}(q)_C$.
\end{proposition}

\subsection{Moduli of semisimple parameters}
\label{sec:moduli-semisimple}
Let ${\hat T}$ be a maximal split torus in ${\hat G}$, and let $W$ be its Weyl group.  Then the quotient ${\hat T}/W$ is a smooth affine
scheme over $\Oc$ of relative dimension the $\Oc$-rank of ${\hat G}$.  If ${\hat G} = GL_n$ and ${\hat T}$ is the standard torus, then we
write an element of ${\hat T}$ as $\diag(x_1, \ldots, x_n)$.  Then ${\hat T} = \Spec\Oc[x_1^{\pm 1}, \ldots, x_n^{\pm 1}]$ and
\[{\hat T}/W = \Spec \Oc[x_1^{\pm 1}, \ldots, x_n^{\pm 1}]^{S_n} =  \Spec \Oc[e_1, \ldots, e_n, e_n^{\pm 1}]\] where $e_i$ is the $i$th
elementary symmetric polynomial in the $x_i$.

\begin{lemma} There is a unique $\Oc$-morphism $\ch : {\hat G} \rarrow {\hat T}/W$ that extends the quotient map ${\hat T} \rarrow
  {\hat T}/W$ and is invariant under conjugation.
\end{lemma}
\begin{proof}
  We can reduce to the case ${\hat G} = GL_n$ and ${\hat T}$ is the standard torus.  Then the map takes $g$ to the point of ${\hat T}/W$ at
  which $e_i$ is the $X^i$-coefficient in the characteristic polynomial of $g$.
\end{proof}
\begin{definition} The $q$-power morphism $q : {\hat T} \rarrow {\hat T}$ takes $t$ to $t^q$.  It descends to a morphism
  \[q : {\hat T}/W \rarrow {\hat T}/W.\]
  We write $({\hat T}/W)^q$ for the fixed-point scheme of $q : {\hat T}/W \rarrow {\hat T}/W$.
\end{definition}
If ${\hat G} = GL_n$ and ${\hat T}$ is standard, we write $q^*e_i$ for the polynomial in the $x_i$ such that
$q^*e_i(x_1, \ldots, x_n) = e_i(x_1^q, \ldots, x_n^q)$, and let \[I_{q,n} \normal \Oc[e_1, \ldots, e_n, e_n^{-1}]\]
be the ideal generated by $(q^*e_i - e_i)_{i=1}^n$.  Then
\[({\hat T}/W)^q = \Spec B_{q,n}\]
for $B_{q,n} = \Oc[e_1, \ldots, e_n, e_n^{-1}]/I_{q,n}$.

\begin{lemma} The fixed-point scheme $({\hat T}/W)^q$ is finite flat over $\Spec \Oc$ and reduced.
\end{lemma}

\begin{proof}
  Again, we assume that ${\hat G} = GL_n$ and ${\hat T}$ is the standard torus.  I claim that
  $B_{q,n} = \Oc[e_1, \ldots, e_n, e_n^{-1}]/I_q$ is generated as an $\Oc$-module by monomials of the form
  $e_1^{a_1}e_2^{a_2}\ldots e_n^{a_n}$ where $0 \leq a_i \leq q - 1$ for all $i$, and $a_n < q-1$.  Granted this, we see
  that $B_{q,n}$ is a finitely generated $\Oc$-module and that
  \[\dim_{\bar{E}} B_{q,n} \otimes_\Oc \bar{E} \leq q^{n-1}(q-1).\] However, the number of $E$-points of $B_{q,n}$ is
  the number of tuples $(z_1, \ldots, z_n)$ of elements of $\bar{E}^\times$ that are permuted by the $q$-power map.
  This number is the same if $\bar{E}^\times$ is replaced by $\bar{k}^\times$; but then it is simply the number of
  semisimple conjugacy classes of $GL_n(k)$, which is seen to be $q^{n-1}(q-1)$ by considering the characteristic
  polynomial.  This shows that the number of $\bar{E}$-points of $B_{q,n}$ is equal to
  $\dim_{\bar{E}} B_{q,n}\otimes \bar{E}$ which is in turn equal to the minimal number of generators of $B_{q,n}$ as an
  $\Oc$-module, whence the result.
  
  To prove the claim, we make an elementary argument with symmetric functions.  If
  $\lambda = (\lambda_1, \lambda_2, \ldots)$ is a partition of a nonnegative integer $|\lambda|$ in which each positive
  integer $j$ appears $a_j = a_j(\lambda)$ times, we let $e_\lambda = \prod_{i=1}^\infty e_{\lambda_i} = \prod_{j=1}^\infty e_j^{a_j}$
  (setting $e_j = 0$ for $j > n$, and $0^0 = 1$). Let $m_\lambda$ be the homogeneous symmetric polynomial in the $x_i$
  of type $\lambda$ (that is, the sum of all monomials of the form $\prod_{i=1}^n x_{\pi(i)}^{\lambda_i}$ for
  $\pi \in S_n$), regarded as an element of the ring $\Oc[e_1, \ldots, e_n]$.  Let $M$ be the $\Oc$-submodule of
  $\Oc[e_1, \ldots, e_n]$ spanned by the set
  \[S = \{e_\lambda : a_j(\lambda) \leq q \text{ for all } 1 \leq i \leq n\}\] and the ideal $I_q$.  Suppose that
  $M \neq \Oc[e_1, \ldots, e_n]$.  Then we may choose $e_\lambda \not \in M$ such that $|\lambda|$ is minimal and such
  that, subject to this, $\lambda$ is maximal with respect to the dominance order $\succ$ on partitions.  By assumption,
  there is some $j$ such that $a_j(\lambda) \geq q$.  Let $\lambda^*$ be the partition such that
  $e_{\lambda^*}e_j^q = e_\lambda$.

  Now, we have \[m_{(q^i)} = q^*e_i \equiv e_i \mod I_q.\] By \cite[Theorem~7.4.4]{stanley1999enumerative},
  $m_{(q^i)} = e_{(i^q)} + \sum_{\mu \succ (i^q)} c_\mu e_\mu$ for some coefficients $c_\mu \in \ZZ$.  Therefore
  \[e_i^q = e_{(i^q)} \equiv e_i - \sum_{\mu \succ (i^q)} c_\mu e_\mu \mod I_q\] and so
  \[ e_\lambda \equiv e_i e_{\lambda^*} - \sum_{\mu \succ (i^q)}c_\mu e_\mu e_{\lambda^*} \mod I_q.\] As $q \geq 2$,
  $e_i e_{\lambda^*} \in M$ by minimality of $|\lambda|$.  Each term $e_{\mu} e_{\lambda^*}$ has the form $e_{\kappa}$
  for a partition $\kappa \succ \lambda$ (depending on $\mu$), and is therefore in $M$ by maximality of $\lambda$.
  Therefore $e_\lambda \in M$, a contradiction.

  Thus $\Oc[e_1, \ldots, e_n]/I_q$ is spanned by those $e_\lambda$ with all $a_j(\lambda) < q$.  In
  $\Oc[e_1, \ldots, e_n, e_n^{-1}]/I_q$ we may replace $q^*e_n - e_n = e_n^q - e_n$ in $I_q$ by $e_n^{q-1} - 1$.  It
  follows that $\Oc[e_1, \ldots, e_n]/I_q$ is spanned by those $e_\lambda$ with all $a_j(\lambda) < q$ and with
  $a_n(\lambda) < q-1$, as required.
\end{proof}

\begin{remark} We do not actually need this result, and in fact it follows from Theorem~\ref{thm:def-rings} below and the
  corresponding facts for $\Xf^{\hat{G}}$.
\end{remark}

If $A$ is an $\Oc$-algebra, let $\Sc^{\hat G}(q)(A)$ be the set of $W$-conjugacy classes homomorphisms
$\tau_s : I_t \rarrow {\hat T}(A)$ such that, for some $w \in W$, $\tau_s(\sigma^q) = \tau_s(\sigma)^w$ for all
$\sigma \in I_t$. Then choosing a generator of $I_t$ shows that the functor $\Sc^{\hat G}(q)$ is represented by an
affine scheme, also denoted $\Sc^{\hat G}(q)$, that is isomorphic to $({\hat T}/W)^q$ (the isomorphism depending on the
choice of generator).  If $C$ is a field containing $\FF$ or $E$, then the $C$-points of $\Sc^{\hat G}(q)$ are in
canonical bijection with the semisimple inertial ${\hat G}$-parameters over $C$.  Restriction to inertia gives a
morphism
\[\ch_{I} : \Xf^{\hat G}(q) \rarrow \Sc^{\hat G}(q).\]

\subsection{Discrete parameters}
\label{sec:discrete}

\begin{definition}
  Let $\tau : I_t \rarrow {\hat G}(C)$ be an extendable homomorphism.  We say that $\tau$ is \emph{discrete} if there is no
  proper Levi subgroup ${\hat M} \subset {\hat G}$ such that $\tau$ factors through an ${\hat M}$-parameter.  We say
  that an inertial $\hat G$-parameter is discrete if every homomorphism in its conjugacy class is.
\end{definition}

\begin{lemma}\label{lem:exists-discrete} If $\tau$ is a representative of an inertial ${\hat G}$-parameter, then there is a Levi subgroup ${\hat M}_\tau$ such
  that $\tau$ factors through a discrete inertial ${\hat M}_\tau$-parameter $\tau : I_t \rarrow {\hat M}_\tau(C)$.
\end{lemma}
\begin{proof} Indeed, simply take $\hat{M}_\tau$ to be a Levi subgroup that is minimal subject to the condition that
  $\hat{M}_\tau(C)$ contains $\tau(I_t)$ and that $\tau : I_t \rarrow \hat{M}(C)$ is extendable.

  Concretely, if $[\zeta] = \{\zeta, \zeta^q, \ldots, \zeta^{q^{r-1}}\}$ is a $q$-power orbit of prime-to-$p$ order
  roots of unity in $C$ and $m \geq 1$ is an integer, let
  \[J_m([\zeta]) = \bigoplus_{i=1}^m J_m(\zeta^{q^i})\]
  (recall from Section~\ref{sec:notation} that $J_m(\zeta^{q^i})$ denotes a Jordan matrix).  Fix a topological generator $\sigma \in I_t$.  Then there is
  some $k \geq 1$ and, for $1 \leq i \leq k$, prime-to-$q$ roots of unity $\zeta_i \in C$ and integers $m_i$, such that
  $\tau(\sigma)$ is conjugate to
  \[\bigoplus_{i=1}^k J_{m_i}([\zeta_i]).\]
  We may then take $\hat{M}_\tau$ to be the standard Levi corresponding to the partition $(r_1m_1, \ldots, r_km_k)$
  where $r_i = |[\zeta_i]|$.
\end{proof}

\subsection{Deformation rings}
\label{sec:deformation}

Let $\rhobar$ be an $\FF$-point of $\Xf^{\hat G}(q)$.  Then the formal completion of $\Xf^{\hat G}(q)$ at $\rhobar$ is
\[X^{\hat G}_{\rhobar} = \Spf R^{\hat G}_{\rhobar}\] where $R^{\hat G}_{\rhobar}$ is the universal framed deformation
ring of $\rhobar$.  The morphism $\Xf^{\hat G}(q) \rarrow \Sc^{\hat G}(q)$ gives an $\FF$-point
$\bar{s} \in \Sc^{\hat G}(q)$, and we let $S^{\hat G}_{\bar{s}}$ be the formal completion of $\Sc^{\hat G}(q)$ at
$\bar{s}$.  Then we have a morphism
\[\ch_I : X^{\hat G}_{\rhobar} \rarrow S^{\hat G}_{\sbar}.\]

\begin{remark}
  Any continuous representation $\rho : W_t \rarrow GL_n(A)$ for a finite ring $A$ has a unique extension to a
  representation of $G_t$.  The deformation ring of $\rhobar$ is therefore the same as the deformation ring of its
  unique extension to $G_t$, which is the object more usually considered. 
\end{remark}

We will compute the local deformation rings at specially chosen points of the special fibre.
\begin{definition}
  let $f \geq 1$ be an integer.  We say that a ${\hat G}$-parameter $\rho : W_t \rarrow {\hat G}(\FF)$ is \emph{$f$-distinguished} if there
  is a Levi subgroup ${\hat M} \subset {\hat G}$ such that $\rho$ factors through an ${\hat M}$-parameter $\rho_{\hat M} : W_t \rarrow {\hat M}(\FF)$ with
  the following properties:
  \begin{enumerate}
  \item $\rho_{\hat M} |_{I_{t}}$ is a discrete inertial parameter;
  \item $Z_{G_\FF}(\rho(\phi^f)_s) \subset {\hat M}_{\FF}$
  \end{enumerate}
  where $Z_{G_\FF}(\rho(\phi^f)_s)$ is the centralizer of $\rho(\phi^f)_s$.

  We say that ${\hat M}$ is an \emph{allowable} Levi subgroup for $\rho$.
\end{definition}
The utility of the second condition is roughly that the eigenspace decomposition of lifts of $\rho(\phi^f)$ may be used
to conjugate lifts of $\rho$ to lie in ${\hat M}$, and so we can reduce to calculating deformation rings for discrete
parameters.

\begin{definition} \label{def:large}
  If $\hat{G}$ has rank $n$, then an integer $f \geq 1$ is \emph{large enough} for $\hat{G}$ if
  \[v_l(q^f - 1) > v_l(n!).\]
\end{definition}

The purpose of the next three sections is to prove the following theorem.  

\begin{theorem} \label{thm:def-rings} Let $f \geq 1$ be large enough for $\hat{G}$, and suppose that
  $\rhobar : W_t \rarrow {\hat G}(\FF)$ is $f$-distinguished.  Let ${\hat M}$ be an allowable Levi subgroup for
  $\rhobar$.  Then there is a formally smooth morphism
  \[\pi:X^{\hat G}_{\rhobar} \rarrow S^{\hat M}_{\bar{s}}\]
  such that the triangle
  \[\begin{tikzcd} X^{\hat M}_{\rhobar} \arrow[rd, "\ch_I"'] \arrow[r,hook] & X^{\hat G}_{\rhobar} \arrow[d, "\pi"] \\
      & S^{\hat M}_{\bar{s}}
  \end{tikzcd}\]
 commutes.
\end{theorem}

The following lemma will be used later to deduce a Breuil--M\'{e}zard-type result.  It is not used in the proof of
Theorem~\ref{thm:def-rings}.

\begin{lemma}\label{lem:unique-cpt} Let $f$ be large enough for $\hat{G}$.  Every irreducible component of $\Xf^{\hat
    G}(q)_\FF$ contains an $f$-distinguished
  $\FF'$-point $\rhobar$ that lies on no other component, for some finite extension $\FF'/\FF$.
\end{lemma}

\begin{proof} Consider an irreducible component labelled by the inertial ${\hat G}$-parameter $\tau$.  Let ${\hat M}$ be
  a Levi subgroup such that $\tau$ factors through a discrete inertial ${\hat G}$-parameter $\tau_{\hat M}$ (one exists,
  by Lemma~\ref{lem:exists-discrete}).  We may extend $\tau$ to an ${\hat M}$-parameter $\rhobar_{\hat M}$, and so a
  ${\hat G}$-parameter $\rhobar$.  Twisting $\rhobar_{\hat M}$ by a sufficiently general element of $Z({\hat M})(\FF')$,
  for some extension $\FF'/\FF$, will ensure that $\rhobar$ is $f$-distinguished with allowable Levi ${\hat M}$.

  That $\rhobar$ lies on a unique irreducible component can be seen directly, but it is easier to appeal to
  Theorem~\ref{thm:def-rings}, which implies that the special fibre of $X^{\hat G}_{\rhobar, \FF'}$ has a unique
  irreducible component since the same is true for $S^{\hat M}_{\bar{s}}$, whose special fibre is local artinian.  As
  the completion map $\Oc_{\Xf^{\hat G}(q)_\FF, \rhobar} \rarrow R^{\hat G}_{\rhobar} \otimes \FF$ is faithfully flat,
  it follows that
  $\Xf^{\hat G}(q)_\FF$ has a unique irreducible component containing $\rhobar$ as required.
\end{proof}

\subsection{Diagonalization}
\label{sec:diagonalization}

\begin{lemma}\label{lem:smoothness} Suppose that $X$, $S$ and $F$ are objects of $\FS_{\Oc}$ and that we
    have morphisms $j : F \rarrow S$, $p :F  \rarrow X$ and $s : X \rarrow F$ such that:
  \begin{enumerate}
  \item $p \circ s = \id_X$; and
  \item $j \circ s \circ p$ is formally smooth.
  \end{enumerate}
  Then $i = j \circ s$ is formally smooth.
\[\begin{tikzcd}
F \arrow[r, "p",swap] \arrow[rd,"j"'] & X \arrow[l, "s", dotted, bend right, swap] \arrow[d, "i = j \circ s"]\\
                              & S      
\end{tikzcd}\]
\end{lemma}
\begin{proof} Define $j' : F \rarrow S$ by $j' = i \circ p = j\circ s \circ p$.  If $F$ and $X$ are made into formal
  schemes over $S$ via $j'$ and $i$ respectively, then $p$ and $s$ are maps of formal schemes over $S$.  Indeed,
  $i \circ p = j'$ by definition, and $j' \circ s = i \circ p \circ s = i$ by the hypothesis that $p \circ s = \id_X$.
  
  Now, as $j'$ is formally smooth by hypothesis, we are (after converting to objects of $\Cc_{\Oc}^\w$ and reversing all
  arrows) in the situation of
  \cite[\href{https://stacks.math.columbia.edu/tag/00TL}{Lemma 00TL}]{stacks-project}, taking into account the remark
  following that lemma. The result follows.
\end{proof}

\begin{lemma}[Diagonalization Lemma] \label{lem:diagonalization} Let $\bar{g} \in {\hat G}(\FF)$ have semisimple part
  $\bar{s}$, and let ${\hat M}$ be a Levi subgroup of ${\hat G}$ such that ${\hat M}_\FF = Z_{G_\FF}(\bar{s})$; note
  that $\bar{g} \in {\hat M}(\FF)$.  Let ${\hat L} \subset {\hat G}$ be a Levi subgroup containing ${\hat M}$.  Let
  $c : {\hat L} \times {\hat G} \rarrow {\hat G}$ be the conjugation map $c(\delta,\gamma) = \gamma \delta \gamma^{-1}$.
  \begin{enumerate}
  \item There is a section
    \[\alpha = \delta \times \gamma: {\hat G}_{\bar{g}}^\w\rarrow {\hat L}_{\bar{g}}^\wedge \times {\hat G}_e^\w\]
    to the completion of $c$ such that the map $\delta : G_{\bar{g}}^\w \rarrow {\hat L}_{\bar{g}}^\w$ is formally smooth.
  \item Suppose that $A \in \Cc_{\Oc}^\wedge$ and that $g \in {\hat L}(A)$ is a lift of $\bar{g}$.  Suppose that $q$ is an
    integer such that $\bar{s}^q$ and $\bar{s}$ are conjugate as elements of ${\hat L}(\FF)$.  Then
    \[ \{h \in {\hat G}(A) : hgh^{-1} = g^q\}  \subset {\hat L}(A).\]
  \end{enumerate}
\end{lemma}
\begin{proof}
  \begin{enumerate}\item
    We may suppose that ${\hat G} = GL_n$ and that ${\hat L} = GL_{n_1} \times \ldots GL_{n_r}$ for some natural numbers $n_i$.  Let
    \[\bar{g} = \begin{pmatrix} X_1 & & \\ & \ddots & \\ & & X_r\end{pmatrix}\]
    for some matrices $X_i \in GL_{n_i}(\FF)$ with characteristic polynomials $\bar{P}_i$. By the assumption that
    ${\hat M} \subset {\hat L}$, the polynomials $\bar{P}_i$ are pairwise coprime. Let $A \in \Cc_\Oc$ and let $g \in {\hat G}(A)$ be a lift
    of $\bar{g}$.  Let $P$ be the characteristic polynomial of $g$.  By Hensel's lemma, $P$ factorises uniquely as a
    product $P = P_1 \ldots P_r$ with each $P_i$ a monic lift of $\bar{P}_i$.  It follows that for each $i$ we may find a
    monic polynomial $R_i$ such that
    \begin{itemize}
    \item $\prod_{j \neq i} P_j \mid R_i$ and
    \item $R_i \equiv I_{n_i} \mod P_i$.
    \end{itemize}
    The matrices $R_i(g)$ are then an orthogonal system of idempotents, and define a direct sum decomposition of $A^n$
    lying above the decomposition of $\FF^n$ associated to ${\hat L}$.  If
    $e^{(1)}_1, \ldots, e^{(1)}_{n_1}, e^{(2)}_{1}, \ldots, e^{(2)}_{n_2}, \ldots,e^{(r)}_{n_r}, \ldots e^{(r)}_{n_r}$
    is the standard basis of $A^n$ then set $f^{(i)}_j = R_i(g)e^{(i)}_j$.  The basis $(f^{(i)}_j)_{i,j}$ is then a
    basis of $A^n$ lifting the standard basis of $\FF^n$ and with respect to which the action of $g$ is a block
    diagonal.  Letting $\gamma$ be the change of basis matrix from $e_j^{(i)}$ to $f_j^{(i)}$, we have that
    $\gamma \in 1 + M_n(\mf_A)$ and $\gamma^{-1} g \gamma \in {\hat L}(A)$.  This construction is functorial and we
    obtain the morphism
    \begin{align*}
      \alpha : {\hat G}^\wedge_{\bar{s}} &\rarrow {\hat L}^\wedge_{\bar{s}} \times {{\hat G}}^\wedge_{e} \\
      g &\mapsto (\delta = \gamma^{-1} g \gamma, \gamma)
    \end{align*} 
    that is evidently a section of $c$.
  
    Let $\pi : {\hat L}^\wedge_{\bar{s}} \times GL_{n, e}^\wedge\rarrow {\hat L}^{\wedge}_{\bar{s}}$ be the projection so that
    \[\delta = \pi \circ \alpha : {\hat G}^\wedge_{\bar{s}} \rarrow {\hat L}^\wedge_{\bar{s}}.\] We will apply Lemma~\ref{lem:smoothness}
    to the diagram
    \[\begin{tikzcd}
        {{\hat L}^{\wedge}_{\bar{s}} \times {GL}^\wedge_{n,e}} \arrow[r, "c"] \arrow[rd, swap, "\pi"]
        &  {\hat G}^{\wedge}_{\bar{s}}\arrow[l, "\alpha", dotted, bend right=49] \arrow[d, "\delta"] \\
        & {\hat L}^{\wedge}_{\bar{s}}
      \end{tikzcd}\] and deduce that $\delta$ is formally smooth, as required.  To apply Lemma~\ref{lem:smoothness} we
    must show that $\delta \circ c$ is formally smooth.  Following carefully through the construction of $\alpha$, one
    finds that this map is
    \[ \delta \circ c : (g, \gamma) \mapsto \gamma_{\hat L} g\gamma_{\hat L}^{-1}\] where $\gamma_{\hat L}$ is the truncation of $\gamma$ obtained by setting
    all of the matrix entries outside of ${\hat L}$ equal to zero.  This is formally smooth: can write it as a composite 
    \[(g, \gamma) \mapsto (g, \gamma_{\hat L}) \mapsto (\gamma_{\hat L}g\gamma_{\hat L}^{-1},\gamma_{\hat L}) \mapsto \gamma_{\hat L} g \gamma_{\hat L}^{-1}\]
    in which the first and third maps are formally smooth, and the second map is an isomorphism.
  \item In the notation of proof of the previous part, the assumption on $\bar{s}$ implies that $R_i(g^q) = R_i(g)$
    for each $i$.  Then any element $h \in {\hat G}(A)$ such that $h^{-1}gh= g^q$ commutes with the projectors $R_i(g)$.  It
    follows that $h$ preserves the direct sum decomposition of $A^n$ associated to the $R_i(g)$; since $g \in {\hat L}$, this
    is exactly the direct sum composition corresponding to ${\hat L}$, whence $h \in {\hat L}(A)$. \qedhere
  \end{enumerate}
\end{proof}
\subsection{Unipotent deformation rings}
\label{sec:unipotent}

Fix standard topological generators $\sigma, \phi$ of $W_t$.  We say that a representation
$\rhobar : W_t \rarrow \hat{G}(\FF)$ is \emph{inertially unipotent} if $\rhobar(\sigma)$ is unipotent --- this is
independent of the choice of $\sigma$.  For this section, we suppose that $\rhobar : W_t \rarrow \hat{G}(\FF)$ is
inertially unipotent, and that it is $f$-distinguished with $\hat{M}$ an allowable subgroup.

If $\hat{G} = GL_{n, \Oc}$, $\rhobar : W_t \rarrow \hat{G}(\FF)$ is a representation that is
$f$-distinguished, inertially unipotent, and ${\hat M}$ is an allowable Levi subgroup for $\rhobar$, then after
conjugating, we may assume that
\begin{equation}\label{eq:standard}\rhobar(\sigma) = \begin{pmatrix} J_{n_1}(1) & & \\ & \ddots & \\ & & J_{n_r}(1)\end{pmatrix}\end{equation}
where $r, n_1, \ldots, n_r \in \NN$, and that the standard Levi subgroup ${\hat M} = \prod_{i=1}^r GL_{n_i}$ is an
allowable subgroup for $\rhobar$.

\begin{lemma}\label{lem:distinguished-phi-implies-sigma}
  Suppose that $A \in \Cc_{\Oc}$ and that $\rho : W_t \rarrow {\hat G}(A)$ is a lift of $\rhobar$ such that
  $\rho(\phi) \in {\hat M}(A)$.

  Then $\rho(\sigma) \in {\hat M}(A)$.
\end{lemma}
\begin{proof}
  This is similar to Lemma~7.9 of \cite{shotton-gln}.  We may and do assume that $\hat{G} = GL_n$ and that $\rhobar$
  and $\hat{M}$ have the form given by equation~(\ref{eq:standard}).  Write $\Sigma = \rho(\sigma)$ and $\Phi = \rho(\phi)$.  By our
  assumptions, we have
  \[\Phi^f = \begin{pmatrix} \Phi_1 & & \\ & \ddots & \\ & & \Phi_{r}\end{pmatrix}\]
  is block diagonal with $\Phi_i \in GL_{n_i}(A)$ for each $i$.  We write
  \[\Sigma = \begin{pmatrix} \Sigma_{11} & \Sigma_{12} & \ldots \\ \Sigma_{21} & \Sigma_{22} & \ldots \\ \vdots & \vdots
      & \ddots \\ \\ \ldots & \Sigma_{r(r-1)} & \Sigma_{rr}\end{pmatrix}\]
  for $\Sigma_{ij} \in M_{n_i \times n_j}(A)$.  Let $I \subset \mf_A$ be the ideal generated by all the entries of all
  $\Sigma_{ij}$ with $i \neq j$.

  We write $\Sigma = 1 + N$ for $N \in M_n(A)$ a lift of a nilpotent matrix.  Then we have
  \begin{align*}
    \Sigma^{q^f} & = (1 + N)^{q^f} \\
    &= 1 + q^fN + \sum_{i=2}^{q^f} \binom{q^f}{i} N^i. 
  \end{align*}
  By the assumption that $f$ is large enough for $\hat{G}$, we have $q^f \equiv 1 \mod \mf_A$ and
  $\binom{q^f}{i} \in \mf_A$ for $1 \leq i \leq n$; by assumption on $\rhobar(\sigma)$ we have
  $N^n \equiv 0 \mod \mf_A$.  We therefore obtain, for each $1 \leq i, j \leq r$, that
  \[(\Sigma^{q^f})_{ij} \equiv \Sigma_{ij} \mod \mf_A I.\]
  However, from the equation $\Phi^f\Sigma = \Sigma^{q^f}\Phi^f$ we get
  \begin{align*}
    \Phi_i \Sigma_{ij} &= (\Sigma^{q^f})_{ij}\Phi_j \\
    &\equiv \Sigma_{ij} \Phi_j \mod \mf_A I.
  \end{align*}
  It follows that \[P(\Phi_i) \Sigma_{ij}\equiv \Sigma_{ij}P(\Phi_j) \mod \mf_AI\] for any polynomial $P \in A[X]$.  If $P_i$ is the
  characteristic polynomial of $\Phi_i$ then, by the assumption that $\rhobar$ is $f$-distinguished, $P_i$ and
  $P_j$ are coprime.  Therefore $P_j(\Phi_i)$ is invertible.  But
  \begin{align*}
    P_j(\Phi_i) \Sigma_{ij} &\equiv \Sigma_{ij}P_j(\Phi_j) \\
    &= 0 \mod \mf_A I
  \end{align*}
  by the Cayley--Hamilton theorem and so $\Sigma_{ij} \equiv 0 \mod \mf_AI$.  As this holds for all $i \neq j$, we see
  that $I \subset \mf_A I$.  By Nakayama's lemma, $I = 0$, so that $\Sigma_{ij} = 0$ for all
  $i \neq j$.  Thus $\Sigma \in {\hat M}(A)$, as required.
\end{proof}

\begin{corollary}\label{cor:reduce-to-M} There is a formally smooth retraction
  \[X_{\rhobar}^{\hat G} \rarrow X_{\rhobar}^{\hat M}.\]
  By retraction, we mean that a left inverse to the natural inclusion $X_{\rhobar}^{\hat M} \rarrow X_{\rhobar}^{\hat G}$.
\end{corollary}
\begin{proof} Let $X_{\rhobar}^{\Phi \in {\hat M}} \subset X_{\rhobar}^{\hat G}$ be the closed sub-formal scheme on
  which $\rho(\phi) \in {\hat M}$.  It follows from Lemma~\ref{lem:diagonalization} part~(1), and the assumption that
  $\rhobar$ is $f$-distinguished with ${\hat M}$ an allowable subgroup, that there is a retraction
  $X_{\rhobar}^{\hat G} \rarrow X_{\rhobar}^{\Phi \in {\hat M}}$.  But Lemma~\ref{lem:distinguished-phi-implies-sigma}
  shows that the inclusion $X_{\rhobar}^{\hat M} \subset X_{\rhobar}^{\Phi \in {\hat M}}$ is actually an equality, and
  the corollary follows.
\end{proof}

In what follows, we denote by $\bar{e}$ the identity point of $\hat{T}(\FF)$, and use the same notation for the
corresponding points of $\hat{T}/W_{\hat{M}}$, $\Sc^{\hat{M}}$, and so on.  Let $S^{\hat{M}}_{\bar{e}}$ be the
completion of $\Sc^{\hat{M}}(q)$ at $\bar{e}$, and for $Z$ any of $\hat{T}$, $\hat{T}/W_{\hat{M}}$ or
    $(\hat{T}/W_{\hat M})^q$ let $Z_{\bar{e}}$ be the completion of $Z$ at $\bar{e}$ (this is perhaps a slight abuse of
    notation).

\begin{theorem}\label{thm:regular-unipotent} Recall our running assumptions that $\rhobar$ is inertially unipotent and
  $f$-distinguished with allowable subgroup $\hat{M}$.

  The map
  \[\ch_{I} : X_{\rhobar}^{\hat M} \rarrow S^{\hat{M}}_1\]
  is formally smooth.
\end{theorem}

\begin{proof} This is an elaboration of the proof of \cite[Proposition 7.10]{shotton-gln}, an argument which is also
  used in \cite[Section 5]{helm2016curtis}.

  We can and do immediately reduce to the case that ${\hat M} = GL_n$.  Then $\rhobar(\sigma)$ is a regular
  unipotent element of ${\hat M}(\FF)$ and we conjugate so that it is equal to the Jordan block $J_n(1)$.

  Let $\hat{T}$ be a split maximal torus in $\hat{M}$.  Our chosen generator $\sigma \in I_t$ identifies
  $S^{\hat M}_{\bar{e}}$ with the $q$-fixed points $(\hat{T}/W_{\hat M})_{\bar{e}}^q$.  Let
  \[Z = \hat{T}_{\bar e} \times_{(\hat{T}/W_{\hat{M}})_{\bar e}} (\hat{T}/W_{\hat M})_{\bar e}^q.\] For $A \in \Cc_\Oc$,
  an $A$-point of $Z$ is the same as a tuple $(t_1, \ldots, t_n)$ of elements of $1 + \mf_A$ such that
  \[\prod_{i=1}^n (X - t_i) = \prod_{i=1}^n (X - t_i^q).\]
  Let $Y$ be the closed formal subscheme of $X^{\hat M}_{\rhobar}$ whose $A$-points are lifts $\rho$ of $\rhobar$ for which
    \[\rhobar(\sigma) =
      \begin{pmatrix}
        a_1 & 1 & 0 & 0 & \ldots \\
        0 & a_2 & 1 & 0 &  \ldots \\
        0 & 0 & a_3 & 1 & \ldots \\
        \vdots & \vdots & \ddots & \ddots & \ddots
      \end{pmatrix}\] for some $a_1, \ldots, a_n \in 1 + \mf_A$.  Then there is a morphism
    \[ Y \rarrow \hat{T}\] taking $\rho$ to $(a_1, \ldots, a_n)$.  Since $\rho(\sigma)$ is conjugate to $\rho(\sigma)^q$, we
    see that this map actually factors through a map $\delta : Y \rarrow Z$.  The diagram
    \[
      \begin{CD}
        Y @>>> X^{\hat G}_{\rhobar} \\
        @VVV @VVV \\
        Z @>>> S^{\hat{M}}_{\bar e}
      \end{CD}\]
    commutes and so we have a morphism $f : Y \rarrow Z \times_{(\hat{T}/W_{\hat M})^q_{\bar e}} X_{\rhobar}^{\hat M}$.  Now I claim:
    \begin{enumerate}
    \item There is a formally smooth morphism of $Z$-formal schemes
      \[s : X^{\hat M}_{\rhobar} \times_{(\hat{T}/W_{\hat M})^q_{\bar e}} Z \rarrow Y.\]
    \item The morphism $\delta : Y \rarrow Z$ is formally smooth.  
    \end{enumerate}
    It follows from these claims, proved below, that the map $\ch_I : X^{\hat M}_{\rhobar} \rarrow (\hat{T}/W_{\hat
      M})^q_{\bar e}$ is formally smooth after
    base change to $Z$.  Since $Z \rarrow (\hat{T}/W_{\hat M})^q_{\bar e}$ is finite flat, this implies (by \cite[Corollaire
    0.19.4.6]{EGA}) that $X^{\hat M}_{\rhobar} \rarrow (\hat{T}/W_{\hat M})^q_{\bar e}$ is formally smooth as required.
\end{proof}

    \begin{proof}[Proof of claim 1] 
      Let $\Pc$ be the completion at the identity of the subgroup $P$ of $\hat{M} = GL_n$ consisting of matrices whose
      first column is $(1, 0, \ldots, 0)^t$.  We have a morphism
      \[\alpha :  Y \times \Pc \rarrow X_{\rhobar}\times_{(\hat{T}/W_{\hat M})^q_{\bar e}} Z\]
      defined by
      \[ \alpha : (\rho, \gamma) \mapsto (\gamma \rho \gamma^{-1}, \delta(\rho)).\]
      We show now that it is an isomorphism.  Define a morphism
      \[ \beta : X_{\rhobar}\times_{(T/W_{\hat M})^q_{\bar e}} Z \rarrow Y \times \Pc\] on $A$-points as follows: suppose given
      an $A$-point $(\rho, (t_1, \ldots, t_n))$ of $(X_{\rhobar}\times_{(T/W_{\hat M})^q_{\bar e}} Z)$; then
      $(T - a_1)(\ldots)(T - a_n) = \ch_{\rho(\sigma)}(T)$. Let $e_1, \ldots, e_n$ be the standard basis for $A^n$ and
      let $f_1, \ldots, f_n$ be defined recursively by:
      \begin{enumerate}
      \item $f_1 = e_1$;
      \item $f_{i + 1} = (\rho(\sigma) - a_i) f_i$.
      \end{enumerate}
      Let $\gamma$ be the matrix (with respect to the standard basis) such that $\gamma(e_i) = f_i$.  Then $\gamma$
      defines a point of $\Pc(A)$, as $f_1 = e_1$ and, by assumption on $\rhobar$, $f_i \equiv e_i \mod \mf_A$.  Note
      that
      \[\rho(\sigma)(f_i) = f_{i+1} + a_if_i\]
      for $1 \leq i \leq n-1$, and
      \begin{align*}\rho(\sigma)f_n &= a_nf_n + (\rho(\sigma) - a_n)f_n \\
                                    &= a_nf_n + \prod_{i=1}^n (\rho(\sigma) - a_n)f_n \\
        &= a_n f_n
     \end{align*}
     by the Cayley--Hamilton theorem and the assumption on $(a_1, \ldots, a_n)$.  It follows that
     $\gamma^{-1} \rho \gamma$ defines an $A$-point of $Y$ lying above the $A$-point $(a_1, \ldots, a_n)$ of
     $Z$.

     We therefore define
     \[\beta\left( \rho, (a_1, \ldots, a_n)\right) = (\gamma^{-1} \rho \gamma, \gamma).\]
     We evidently have $\alpha \circ \beta = \id$, and one checks directly from the constructions that
     $\beta \circ \alpha = \id$.  So $\alpha$ and $\beta$ are isomorphisms, as required.  The map $s$ of claim (1) is
     then just the composition of $\beta$ with projection to $Y$.
    \end{proof}

    \begin{proof}[Proof of claim 2]
      Let $Y \rarrow Z \times (\AA^n)_{\bar{e}_1}^\wedge$ be the morphism $\rho \mapsto (\delta(\rho), \rho(\phi)(e_1))$.
      I claim that this is an isomorphism.  To see injectivity (at the level of $A$-points), note that for $i \geq 2$ we
      can recover $\rho(\phi)(e_i)$ inductively from the formula
      \begin{align*}\rho(\phi)(e_{i+1}) &= \rho(\phi)(\rho(\sigma) - a_i)(e_i) \\
                                        &= (\rho(\sigma)^q - a_i)\rho(\phi)(e_i).
      \end{align*} 
      For surjectivity, note that the above inductive formula certainly determines a lift $\Phi$ of $\rhobar(\phi)$ with
      given $\Phi(e_1)$, and we have only to check that $\Phi \rho(\sigma) = \rho(\sigma)^q \Phi$ holds.  For $i < n$,
      we have
      \begin{align*}\Phi \rho(\sigma)(e_i) &= \Phi (a_i e_i + e_{i+1}) \\
                                           &= \Phi(a_ie_i) + (\rho(\sigma)^q - a_i) \Phi (e_i) \\
        &= \rho(\sigma)^q \Phi(e_i)
          \end{align*}
    as required.  For $i = n$, note that (writing $\Sigma = \rho(\sigma)$)
    \begin{align*}
      (\rho(\sigma^q) - a_n)\Phi(e_n) 
      &= (\Sigma^q - a_n)(\Sigma^q - a_{n-1}) \Phi (e_{n-1})\\
      &= \ldots \\
      &= (\Sigma^q - a_n)(\Sigma^q - a_{n-1})( \ldots)(\Sigma^q - a_1) \Phi(e_1) \\
      &= \ch_{\Sigma}(\Sigma^q)\Phi(e_1) \\ 
      &= \ch_{\Sigma^q}(\Sigma^q)\Phi(e_1) \\\intertext{(by our assumption on $(a_1, \ldots, a_n)$)}
      &= 0.
    \end{align*}It follows that
    \[\Phi \Sigma (e_n) = \Phi (a_ne_n) = \Sigma^q \Phi (e_n),\]
    as required.
    \end{proof}

    \begin{corollary} \label{cor:unipotent-case}
      Let $\rhobar$ and ${\hat M}$ be as above.  Then there is a formally smooth morphism
      \[X_{\rhobar}^{\hat G} \rarrow S^{\hat M}_{\bar e}\] whose composition with the inclusion
      $X_{\rhobar}^{\hat M} \into X_{\rhobar}^{\hat G}$ is $\ch_I$.
    \end{corollary}
    \begin{proof}
      Immediate from Corollary~\ref{cor:reduce-to-M} and Theorem~\ref{thm:regular-unipotent}.
    \end{proof}

\subsection{Reduction to the unipotent case}
\label{sec:unipotent-reduction}

We explain how to deduce Theorem~\ref{thm:def-rings} from the inertially unipotent case
(Corollary~\ref{cor:unipotent-case}).  The argument is essentially that of \cite{ClozelHarrisTaylor2008-Automorphy}
Corollary~2.13 and \cite{MR3622123} Proposition~2.6, albeit phrased slightly differently.

Fix standard topological generators $\sigma, \phi$ of $W_t$.  Suppose that ${\hat G}$ is as above, that ${\hat M}$ is a
Levi subgroup containing a split maximal torus $\hat{T}$, and that $f$ is large enough for $\hat{G}$.  Let
$n = \rk(\hat{G})$. We impose the following assumption on the $l$-adic coefficient system $(E, \Oc, \FF)$:
\begin{equation}
  \label{eq:sufflarge} 
  \text{$E$ contains the $(q^{n!}-1)$th roots of unity.}
\end{equation}

Suppose that
$\rhobar : W_t \rarrow {\hat G}(\FF)$ is $f$-distinguished with allowable subgroup ${\hat M}$.  Write
$\rhobar|_{I_t} = \tau_s \tau_u$ with $\tau_s$ semisimple and $\tau_u$ unipotent.  Up to conjugation, using the
assumption~\eqref{eq:sufflarge}, we may and do assume that $\tau_s$ has image in
${\hat T}(\FF)$.  Let $\tilde{\tau}_s$ be the unique lift of $\tau_s$ to ${\hat T}(\Oc)$ having order coprime to $l$.

First, we reduce to the case that the eigenvalues of $\tau_s(\sigma)$ form a single orbit under the $q$-power map.  Let
\[{\hat L}_0 = \{g \in {\hat G} : g \tilde{\tau}_s g^{-1} = \tilde{\tau}_s^{q^i} \text{for some $i \in \NN$}\},\] so
that ${\hat L}_0 = Z_{\hat G}(\tilde{\tau}_s) \rtimes \langle w \rangle$ for some element $w$ of the Weyl group $W$.
Finally, let
\[ {\hat L} = Z_{\hat G}(Z({\hat L}_0)),\] a Levi subgroup of ${\hat G}$.  Then certainly
$Z_{\hat G}(\tilde{\tau}_s) \subset {\hat L}$. By Lemma~\ref{lem:diagonalization}~(1), there is a morphism
$\gamma:{\hat G}_{\rhobar(\sigma)}^\wedge \rarrow {\hat G}_e^\w$ such that conjugating by $\gamma(\rho(\sigma))$ defines a
formally smooth morphism
\begin{align*}X_{\rhobar}^{\hat G} &\rarrow X_{\rhobar}^{\sigma \in {\hat L}} \\
  \rho &\mapsto \gamma(\rho(\sigma))^{-1} \rho \gamma(\rho(\sigma))
\end{align*}
where the space on the right is the closed formal subscheme of $X^{\hat{G}}_{\rhobar}$ on which
$\rho(\sigma) \subset {\hat L}$ (which is clearly independent of the choice of $\sigma$). By part (2) of the same Lemma,
\[X_{\rhobar}^{\sigma \in {\hat L}} = X_{\rhobar}^{\hat L}.\] It is therefore enough to prove Theorem~2.13 with
${\hat G}$ replaced by ${\hat L}$; note that $\rhobar$ is still $f$-distinguished as a representation valued in
${\hat L}$.  Since ${\hat L}$ is a product of general linear groups, it in fact suffices to prove Theorem~2.13 in the
case that ${\hat G} = {\hat L} = GL_n$ for some $n$.  Then we have that $Z({\hat L}_0) = Z({\hat G})$, which happens if
and only if the eigenvalues of $\tau_s(\sigma)$ form a single orbit under the $q$-power map (for any $\sigma$).  So, up
to conjugating $\rhobar$, we may assume that $n = rd$ for some integers $r$ and $d$, where $d$ is the smallest
natural number with $\tau_s^{q^d} = \tau_s$, and that
\begin{equation}\tau = \diag(\tau_r, \tau_r^{q}, \ldots, \tau_r^{q^{d-1}})\label{eq:tau-std}\end{equation}
for some homomorphism $\tau_r : I_t \rarrow GL_r(\FF)$ with scalar semisimplification.  From now on we assume $\tau$ has
this form.  We also regard $GL_r$ as being embedded in $GL_n$ in the `top left corner'.

Let $W_t^{(d)}$ be the subgroup of $W_t$ generated by $I_t$ and $\phi^d$.  Our next step is to show that deforming
$\rhobar$ is the same as deforming the `top-left part' of the restriction to $W_t^{(d)}$.

Let
\[{\hat N} = Z_{\hat G}(\tilde{\tau}_s).\] Then $\hat{N}$ is the standard Levi subgroup with block sizes
$(r, r, \ldots, r)$.  Let $\pi : {\hat N} \rarrow GL_r$ be the map that forgets the entries outside of the first copy of
$GL_r \subset {\hat N}$.  Choose $w \in W$ such that $\tau_s^{q} = w\tau_s w^{-1}$ and such that $w^d = e$.  Specifically, with the above form of
$\tau$ we can take $w$ to be the block matrix (with $r \times r$ blocks)
\[w =
  \begin{pmatrix} 0 & I & 0 & \ldots & 0 \\ 0 & 0 & I & \ldots & 0 \\ \vdots & \vdots & \vdots & \ddots
    & \vdots \\
    0 & 0 & 0 & \ldots & I \\
    I & 0 & 0 & \ldots & 0
  \end{pmatrix}.\] Then $\rhobar(W_t^{(d)}) \subset {\hat N}(\FF)$.  Let
$X_{\rhobar}^{\sigma \in {\hat N}} \subset X^{\hat G}_{\rhobar}$ be the closed formal subscheme on which
$\rho(\sigma) \subset {\hat N}$ (this is clearly independent of the choice of $\sigma$).  Then
Lemma~\ref{lem:diagonalization} implies that there is a formally smooth retraction
\[X_{\rhobar}^{\hat G} \rarrow X_{\rhobar}^{\sigma \in {\hat N}}\] to the natural inclusion, and that $\rho(\phi) \in w{\hat N}$ for all
$X_{\rhobar}^{\hat G}$.  If $\rho : W_t \rarrow {\hat N}(A) \rtimes \langle w \rangle$ is a continuous representation, then we write
$\rho^{(d)}$ for the representation
\[\pi \circ \rho|_{W\up{d}_t} : W\up{d}_t \rarrow GL_r(A).\] 

\begin{lemma}
  The map \[\rho \mapsto \rho\up{d}\] defines a formally smooth morphism
  $X_{\rhobar}^{\sigma \in {\hat N}} \rarrow X^{GL_r}_{\rhobar\up{d}}$.
\end{lemma}

\begin{proof} Let $A \in \Cc_\Oc$.  For $g \in {\hat N}(A)$ any element, let $g_i$ be the projection onto the $i$th
  factor of ${\hat N}$ (so $g_i \in GL_{r}(A)$).  If $\rho$ is an $A$-point of $X_{\rhobar}^{\sigma \in {\hat N}}$, we
  write $\Sigma$ and $\Phi$ for $\rho(\sigma)$ and $\rho(\phi)$.  Any point of $X_{\rhobar}^{\sigma \in {\hat N}}(A)$
  has the form $(\Sigma,\Phi = w\Psi)$ for $\Sigma, \Psi \in {\hat N}(A)$ such that
  $\Psi_i\Sigma_i\Psi_i^{-1} = \Sigma_{i-1}^q$ for all $i$ (with indices taken modulo $d$).  Note that
  $(\Phi^d)_1 = \Psi_2 \ldots \Psi_d \Psi_1$.  Define a morphism
  \begin{align*}X^{\sigma \in {\hat N}}_{\rhobar} & \rarrow X^{GL_r}_{\rhobar\up{d}} \times \prod_{i=2}^{d}GL^\w_{r, \bar{\Psi}_i} \\
    (\Sigma, w\Psi) & \mapsto \left((\Sigma_1, (w\Psi)^d_1), \Psi_2, \ldots, \Psi_d\right).
  \end{align*}
  This is in fact an isomorphism; we may write down the inverse
  \[\left((\Sigma\zeta^{-1}, \Phi), \Psi_2, \ldots, \Psi_d\right) \mapsto (\Sigma', w\Psi') \]
  where $\Sigma'$ is defined by $\Sigma'_1 = \Sigma$ and $\Sigma'_i = \Psi_i^{-1} (\Sigma'_{i-1})^q\Psi_i$ for
  $i \geq 2$, and $\Psi'$ is defined by $\Psi'_i = \Psi_i$ for $i \geq 2$ and $\Psi'_1 = (\Psi_2\ldots\Psi_d)^{-1}\Phi$.
  The lemma follows.
\end{proof}

We therefore have a formally smooth map
\[ X^{\hat G}_{\rhobar} \rarrow X^{GL_r}_{\rhobar\up{d}}.\] If we let ${\hat M}' = {\hat M} \cap GL_r$, then we may redo
the above arguments with ${\hat G}$ replaced by ${\hat M}$ and $GL_r$ replaced by ${\hat M}'$ and obtain a commuting
diagram
\begin{equation} \label{cd1}\begin{tikzcd} X^{\hat M}_{\rhobar} \arrow[d, hook] \arrow[r] & X^{{\hat M}'}_{\rhobar\up{d}} \arrow[d,hook] \\
    X^{\hat G}_{\rhobar} \arrow[r] & X^{GL_r}_{\rhobar\up{d}}
\end{tikzcd}\end{equation}
in which the horizontal morphisms are formally smooth.

The representation $\rhobar\up{d} : W_t\up{d} \rarrow GL_r(\FF)$ has the property that $\rhobar\up{d}|I_t$ has
semisimplification given by a scalar $\bar{t} : I_t\up{d} \rarrow Z(GL_r(\FF))$.  Choose an extension of $\bar{t}$ to
$W_t\up{d}$ and let $\theta : W_t\up{d} \rarrow Z(GL_r(\Oc))$ be its Teichm\"{u}ller lift.  Twisting by $\theta$ gives a
bijection between deformations of $\rhobar\up{d}$ and deformations of $\rhobar \up{d} \otimes \theta^{-1}$, which is
unipotent on inertia.  We can therefore apply Corollary~\ref{cor:unipotent-case}, which shows that there is a formally
smooth morphism $X^{GL_r}_{\rhobar\up{d}} \rarrow \Sc^{\hat M'}(q^d)_{\bar{t}}$ such that the triangle
\begin{equation}\label{cd2}\begin{tikzcd} X^{{\hat M}'}_{\rhobar\up{d}} \arrow[d,hook] \arrow[dr,"\ch_I"] & {}  \\
    X^{GL_r}_{\rhobar\up{d}} \arrow[r]& \Sc^{{\hat M}'}(q^d)_{\bar{t}}.
\end{tikzcd}\end{equation}
commutes.

We may choose an inclusion $\hat{M}' \times \ldots \times \hat{M}' \into \hat{M}$ as a normal subgroup, where there are $d$
copies of $\hat{M}'$, such that conjugation by $\rhobar(\phi) \in \hat{M}$ permutes these copies cyclically.  Take
$\hat{T}'$ to be a split maximal torus of $\hat{M}'$ and $\hat{T} = \hat{T}' \times \ldots \times \hat{T}'$ the split
maximal torus of $\hat{M}$ obtained from it.  The map
\[(\bar{t}, \bar{t}^q, \ldots, \bar{t}^{q^{d-1}}) : I_t \rarrow Z(\hat{M}' \times \ldots \times \hat{M}')(\FF) \into \hat{T}(\FF)\]
defines a point $\bar{s}$ of $\Sc^{\hat M}(q)(\FF)$ which is exactly the point corresponding to $\rhobar|_{I_t}$.  
\begin{lemma} There is an isomorphism
  \[S^{\hat{M}}_{\bar{s}} = \Sc^{\hat M}(q)_{\bar{s}} \isomto \Sc^{{\hat M}'}(q^d)_{\bar{t}}\] such that the diagram
 \begin{equation}\label{cd3}\begin{tikzcd}X^{\hat M}_{\rhobar}  \arrow[d] \arrow[r,"\ch_I"] & S^{\hat M}_{\bar{s}} \arrow[d,equal ]  \\
 X^{{\hat M}'}_{\rhobar\up{d}} \arrow[r, "\ch_I"]& \Sc^{{\hat M}'}(q^d)_{\bar{t}}
\end{tikzcd}\end{equation}
commutes.
\end{lemma}
\begin{proof}
  We write down the map on $A$-points.  This sends the $W_{\hat M}$-orbit of $(s_1, s_2, \ldots, s_r)$, where each
  $s_i : I_t \rarrow \hat{T}'(A)$ is a lift of $\bar{s}$, to the $W_{{\hat M}'}$-orbit of
  $s_1$.  This is an isomorphism; its inverse is the map taking the
  $W_{{\hat M}'}$-orbit of $s_1$ to the $W_{\hat M}$-orbit of
  \[(s_1, s_1^q, \ldots, s_1^{q^{d-1}}).\qedhere\]
\end{proof}

\begin{proof}[Proof of Theorem~\ref{thm:def-rings}] Putting the commuting diagrams \eqref{cd1}, \eqref{cd2} and~\eqref{cd3}
  together, we obtain a commuting triangle
\[\begin{tikzcd} X^{\hat M}_{\rhobar} \arrow[rd, "\ch_I"'] \arrow[r,hook] & X^{\hat G}_{\rhobar} \arrow[d] \\
      & S^{\hat M}_{\bar{s}}
    \end{tikzcd}\] in which the right hand vertical morphism is formally smooth, as required.
\end{proof}

\section{Representations of finite general linear groups}
\label{sec:rep-gln}

\subsection{Dual groups, tori and parameters}
\label{sec:dual}

We follow \cite{MR2480618} section~4.3 and give a formulation of Deligne--Lusztig theory that is adapted for our
purposes.

Recall that $k$ is the residue field of $F$, of order $q$.  Let $G$ be a product of general linear groups over $k$, and
let $\TT$ be a split maximal torus of $G$ defined over $k$.  We fix an $l$-adic coefficient system $(E, \Oc, \FF)$.  We
take $\hat T$ and $\hat G$ to be a dual torus of $\TT$ and dual group of $G$, defined over $\Oc$.  We assume that $E$ is
sufficiently large; precisely, we impose the assumption~\eqref{eq:sufflarge}.  We write $X = X(\TT) = \Hom(\TT, \GG_m)$,
$Y = Y(\TT) = \Hom(\GG_m, \TT)$, $X(\hat{T}) = \Hom(\hat T, \GG_m)$, and
$Y(\hat T) = \Hom(\GG_m, \hat{T})$.

By definition, we have fixed isomorphisms
\begin{align*}X(\TT) &= Y(\hat T) \\ \intertext{and} Y(\TT) & = X(\hat T) \end{align*} respecting
the natural pairings.

We write $W = W(G,\TT)$ for the Weyl group of $\TT$.  It acts on the left on $\TT$.  We thus obtain left actions on
$X(\TT)$ and $Y(\TT)$: the former is defined by $w\alpha = \alpha \circ w^{-1}$ and the latter by
$w\beta = w \circ \beta$, for all $\alpha \in X(\TT)$, $\beta \in Y(\TT)$, $w \in W$.  Thus $W$ acts on the left on
$Y(\hat T)$ and $X(\hat T)$.  Let $\hat{W} = W(\hat{G}, \hat{T})$.  Then there is an isomorphism
$\delta : W \isomto \hat{W}$ such that the action of $w$ on $X(\TT)$ agrees with the action of $\delta(w)$ on
$Y(\hat{T})$.  We identify $W$ with $\hat{W}$ along this isomorphism.  Note that this is differs from the
anti-isomorphism of \cite{MR2480618} by an inverse; we find it more convenient to work with a group isomorphism.

Now let $T \subset G$ be another maximal torus, not necessarily split.  Choose $g \in G(\bar{k})$ such that
$T_{\bar{k}} = g\TT_{\bar{k}}g^{-1}$.  Then $g^{-1}F(g) \in N(\TT_{\bar{k}})$; write $w$ for its image in $W$.  This
induces a bijection between $G(k)$-conjugacy classes of maximal tori in $G$, and conjugacy classes in $W$.  If $w$ is
any element of $W$, we write $T_w$ for a choice of torus in the corresponding conjugacy class.  If $F$ is the geometric
Frobenius morphism over $k$, then the diagram
\[ \begin{CD} \TT_{\bar{k}} @>{\ad_g}>>T_{\bar{k}} \\ @V{wq}VV @V{F}VV \\ \TT_{\bar{k}}@>{\ad_g}>> T_{\bar{k}}\end{CD}\]
commutes.  Consequently, $\ad_g$ induces an isomorphism $\TT(\bar{k})^{wq} \isomto T(k)$.  Choose $n$ such that
$w^n = e$ and write $N = 1 + qw + (qw)^2 + \ldots + (qw)^{n-1} \in \ZZ[W]$.  Then there is an isomorphism
\begin{equation}\label{eq:normisom}N :\TT(k_n)/(1 - qw) \isomto \TT(\bar{k})^{wq}.\end{equation}

Recall that $E$ satisfies assumption~\eqref{eq:sufflarge}.  Then we have isomorphisms
\begin{align}\Hom(\TT(k_n), E^\times) &\cong \Hom(Y \otimes k_n^\times, E^\times) \label{line1} \\
                                    & \cong \Hom(k_n^\times, \Hom(Y,E^\times)) \label{line2} \\
                                    & \cong \Hom(k_n^\times, \hat{T}(E)), \label{line3}
\end{align}
the first isomorphism coming from $\TT(k_n) = Y\otimes k_n^\times$ and the last from
\[\hat{T}(E) \cong \Hom(X(\hat{T}), E^\times) = \Hom(Y, E^\times).\]
The composite of the isomorphisms~\eqref{line1}--\eqref{line3} takes $\theta \in \Hom(\TT(k_n), E^\times))$ to the
element $s \in \Hom(k_n^\times, \hat{T}(E))$ such that
\[ y(s(\alpha)) = \theta(y(\alpha))\] for all $y \in Y(\TT) = X(\hat{T})$ and $\alpha \in k^\times_n$.  Combining with the
isomorphism $N$ from equation~(\ref{eq:normisom}), we obtain an isomorphism
\[\Hom(\TT(\bar{k})^{wq}, E^\times) \cong \Hom(k_n, \hat{T}(E)^{w = q}).\]
Finally, we compose with the natural surjection $I_t \onto k_n$ and note that every homomorphism $I_t \rarrow
\hat{T}(E)^{w = q}$ factors through this surjection, so that we have an isomorphism
\begin{equation}\label{eq:isom}
    \Hom(\TT(k_n)^{wq}, E^\times) \cong \Hom(I_t, \hat{T}(E)^{w = q})\end{equation}
that is independent of any choices (of generators for $I_t$ or $k_n^\times$, or groups of roots of unity in $E$).  If we
choose, additionally, $n$ to be large enough that $g \in G(k_n)$, and compose the isomorphism~\eqref{eq:isom} with the isomorphism
$\ad_g : \TT_{k_n} \rarrow T_{k_n}$, we get
\[\Hom(T(k), E^\times) \cong \Hom(I_t, \hat{T}(E)^{w = q}).\]

\begin{remark} This isomorphism is exactly the restriction to tame inertia of the local Langlands correspondence for
  unramified tori constructed in \cite[section 4.3]{MR2480618} (over the complexes, but the construction works equally
  well over any field of characteristic zero containing enough roots of unity).
\end{remark}

We therefore obtain, for every $T$ and every $\theta \in \Hom(T(k), E^\times)$, a $W$-conjugacy class of pairs $(w, s)$
where $w \in W$ and $s : I_t \rarrow \hat{T}(E)^{w = q}$.  Then it is easy to check the following lemma.

\begin{lemma}\label{lem:dual-conjugacy} The above map taking $(T, \theta)$ to $(w, s)$ gives a bijection between
  \[\{\text{conjugacy classes of pairs $(T, \theta)$ : $T$ maximal torus in $G$, $\theta \in \Hom(T(k),E^\times)$}\}\] and
  \[\{\text{$W$-conjugacy classes of $(w,s)$ : $w \in W$ and $s \in \Hom(I_t, \hat T(E)^{w = q})$}\}.\qedhere\]
\end{lemma}

Recall (see for example \cite{MR1118841} Definition~13.2) that two pairs $(T, \theta)$ and $(T', \theta')$ are
\emph{geometrically conjugate} if there is some $n \geq 1$ and $h \in G(k_n)$ such that $T'_{k_n} = hT_{k_n}h^{-1}$ and
\[\theta \circ N_{k_n/k} = \theta' \circ N_{k_n/k} \circ \ad_h\]
as characters of $T(k_n)$, where $N_{k_n/k}$ is the norm.

\begin{lemma}\label{lem:dual-geometric-conjugacy}
  The above map $(T, \theta) \mapsto s$ induces a bijection between
  \[\{\text{geometric conjugacy classes of pairs $(T, \theta)$}\}\]
  and
  \[\{\text{$q$-power stable $W$-orbits of $s \in \Hom(I_t, \hat T(E))$}\}.\]
\end{lemma}
\begin{proof}
  Let $n$ be such that $w^n = 1$ for all $w \in W$.  If $T$ is a maximal torus of $G$ and $g \in G(k_n)$ is such that
  $T_{\bar{k}} = g\TT_{\bar{k}}g^{-1}$ and if $w$ is the class of $g^{-1}F(g)$ in $W$, and $N = 1 + qw + \ldots +
  (qw)^{n-1} \in \ZZ[W]$, then we have a commuting diagram
  \[
    \begin{CD} 
      \Hom(T(k), E^\times) @>{\circ \ad_g}>>\Hom(\TT(k_n)^{wq},E^\times) @>>>\Hom(I_t, \hat{T}(E)^{w = q})\\
      @V{\circ N_{k_n/k}}VV @V{\circ N}VV @VVV \\
      \Hom(T(k_n), E^\times) @>{\circ \ad_g}>> \Hom(\TT(k_n), E^\times) @>>> \Hom(I_t, \hat{T}(E)[q^n - 1]).
    \end{CD}
  \]
  The rightmost horizontal arrows are as above, while the rightmost vertical arrow is the obvious inclusion.  Hence
  geometric conjugacy classes of pairs $(T, \theta)$ are in bijection with $q$-power stable $W$-orbits of
  $s \in \Hom(I_t, \hat{T}(E))$ (note that such $s$ automatically has image in $\hat{T}(E)[q^n-1]$).  We see that two
  pairs $(T, \theta)$ and $(T', \theta')$ are geometrically conjugate if and only if the corresponding homorphisms $s$
  and $s'$ are in the same $W$-orbit.  Thus the map taking the geometric conjugacy class of $(T, \theta)$ to the
  $W$-orbit of $s$ is well-defined and injective.  It is surjective by Lemma~\ref{lem:dual-conjugacy}.
\end{proof}

\subsection{Representations of $G(k)$}
\label{sec:DL-reps}

If $s \in \Hom(I_t, \hat T(E))$ is $W$-conjugate to its $q$-th
power, we write $W(s)$ for the stabiliser of $s$ and
\[W(s,s^q) = \{w \in W: {}^ws = s^q\}.\] Thus $W(s,s^q)$ is a left coset of $W(s)$ in $W$.  Note also that
$W(s) = W(s^q)$, so that $W(s)$ acts on $W(s, s^q)$ by conjugation.  Let $\epsilon : W \rarrow \pm 1$ be the sign
character.  For a field $C$ we write $K_C(G(k))$ for the Grothendieck group of representations of $G(k)$ over $C$.

\begin{definition}[Deligne--Lusztig representations] Let $(w, s)$ be a pair comprising an element $w$ of $W$ and a
  homomorphism $s \in \Hom(I_t, \hat T(E)^{w = q})$.  Then we define a virtual representation $R(w,s)$ of $G(k)$ by
  \[R(w,s) = R^{\theta}_{T}\] where $(T,\theta)$ corresponds to $(w, s)$ as in Lemma~\ref{lem:dual-conjugacy}.
  Here $R^\theta_T$ is the Deligne--Lusztig virtual representation constructed in \cite{MR0393266}.
\end{definition}

\begin{definition}[generalized Steinberg representations] Let $s$ be an element of $\Hom(I_t, \hat{T}(E))$,
  $W$-conjugate to its $q$th power.  Define an element \[\pi_G(s) \in K_E(G(k)) \otimes \QQ\] by
\[\pi_G(s) = |W(s)|^{-1}\sum_{w \in W(s,s^q)} \epsilon(w)R(w,s).\]
\end{definition}

\begin{proposition}
   The element $\pi_G(s) = K_E(G(k)) \otimes \QQ$ is (the class of) an irreducible representation.
\end{proposition}

\begin{proof}
  This follows from \cite{MR0393266} Theorem~10.7 (i).  The formula there states that that
  \begin{equation}\label{DL-form} \sum_{(T, \theta) \bmod G(k)} \frac{(-1)^{\rk_k(G) - \rk_k(T)}}{\pres{R^\theta_T,
        R^\theta_T}}R^\theta_T\end{equation}
  is the class of an irreducible representation, where the sum is over all $G(k)$-conjugacy classes of $(T, \theta)$
  in the geometric conjugacy class of $s$ (under the correspondence of Lemma~\ref{lem:dual-geometric-conjugacy}).

  We claim first that, if $T$ is a maximal torus of $G$ corresponding to $w \in W$, then
  \[(-1)^{\rk_k(G) - \rk_k(T)} = \epsilon(w).\] Indeed, $\rk_k(T)$ is the dimension of the $(+1)$-eigenspace of $w$
  acting on $X(\TT) \otimes \CC$.  Since the eigenvalues of $w$ occur in conjugate pairs, this has the same parity as the
  difference of $\rk_k(G) = \dim X(\TT) \otimes \CC$ and the dimension $d$ of the $(-1)$-eigenspace.  As
  $\epsilon(w) = \det(w | X(\TT)) = (-1)^{d}$, we obtain the claim.

  We claim next that $\pres{R^\theta_T, R^\theta_T} = |Z_{W}(w) \cap W(s')|$ if $(T, \theta)$ corresponds to $(w,s')$.
  Indeed, we have the formula (\cite{MR0393266} Theorem~6.8)
  \[\pres{R^\theta_T, R^\theta_T} = |\{v \in W(T)^F : {}^v\theta = \theta\}|.\]
  The identification of $W(T)$ with $W(\TT) = W$ via $\ad_g$ identifies $W(T)^F$ with $Z_W(w)$ and the stabiliser of
  $\theta$ with the stabiliser of $s'$, and we have
  \[\pres{R^\theta_T, R^\theta_T} = |\{v \in Z_W(w) : {}^vs' = s'\}| = |Z_{W}(w) \cap W(s')|\]
  as required.

  We now can rewrite the expression~\eqref{DL-form} as
  \[\sum_{(w,s') \bmod W} \frac{\epsilon(w)}{|Z_W(w) \cap W(s')|} R(w,s')\]
  where the sum runs over $W$-conjugacy classes of pairs $(w,s')$ such that $s'$ is $W$-conjugate to $s$ and
  $w \in W(s', (s')^q)$.  We can conjugate each term $(w, s')$ in this sum so that $s' = s$ and rewrite it as
  \[\sum_{w \in W(s, s^q) \bmod W(s)} \frac{\epsilon(w)}{|Z_W(w) \cap W(s)|} R(w,s)\]
    where the sum is over $W(s)$-conjugacy classes in $W(s,s^q)$.  Finally, we rewrite this as
    \[\frac{1}{|W(s)|}\sum_{w \in W(s,s^q) \bmod W(s)} \frac{|W(s)|}{|Z_W(w) \cap W(s)|} \epsilon(w)R(w,s),\]
    which on application of the orbit-stabiliser theorem (to the conjugation action of $W(s)$ on $W(s,s^q)$) becomes
    \[\frac{1}{|W(s)|}\sum_{w \in W(s,s^q)} \epsilon(w)R(w,s),\]
    as required.
\end{proof}

\begin{definition}\label{def:piGtau} Suppose that $\tau : I_t \rarrow \hat{G}(E)$ is an inertial $\hat G$-parameter, and
  assume that $\tau_s$ has image in $\hat{T}(E)$.  Then there is a split Levi subgroup $L \subset G$, with dual Levi
  $\hat{L}\supset \hat{T}$, such that $\tau$ factors through a discrete inertial $\hat{L}$-parameter.  Define a
  representation $\pi_G(\tau)$ of $G$ by
  \[\pi_G(\tau) = \Ind_{L(k)}^{G(k)} \pi_L(s)\]
  and note that this is (up to isomorphism) independent of the choice of $L$.
\end{definition}

Next we recall some facts about the Gelfand--Graev representation.  Let $B$ be a Borel subgroup of $G$ containing
the split maximal torus $\TT$, and let $U$ be its unipotent radical.  Let $\psi : U(k) \rarrow W(\FF)^\times$ be a
character in general position (i.e. whose stabiliser in $B/U$ is $ZU/U$).
\begin{definition} The (integral) Gelfand--Graev representation is \[\Gamma_G = \Ind_{U(k)}^{G(k)}\psi.\]
  Up to isomorphism, it is independent of the choices of $T, B$, and $\psi$.
\end{definition}
If $A$ is a $W(\FF)$-algebra then we set $\Gamma_{G,A} = \Gamma_G \otimes_{W(\FF)} A$.
\begin{lemma}\label{lem:projective} For any $W(\FF)$-algebra $A$, the representation $\Gamma_{G, A}$ is a projective $A[G(k)]$-module.
\end{lemma}
\begin{proof}
  By Frobenius reciprocity, it suffices to show that $A$, with the action of $U(k)$ via $\psi$, is a projective
  $A[U(k)]$-module.  This is true as $|U(k)|$ is invertible in $W(\FF)$.
\end{proof}

\begin{theorem}\label{thm:gelfand-mult} The representation $\Gamma_{G, E}$ is multiplicity-free, and
  \[\Gamma_{G, E} \cong \bigoplus_{[s]} \pi_{G}(s)\]
  where $[s]$ runs over the $q$-power stable $W$-orbits of $\Hom(I_t, \hat{G}(E))$.
\end{theorem}
\begin{proof} This is \cite{MR0393266} Theorem~10.7~(ii).
\end{proof}

The final lemma of this section is only needed to compare this article with \cite{shotton-gln}.

\begin{lemma}\label{lem:comparison} Suppose that we are in the situation of Definition~\ref{def:piGtau}, that
  $\rho : W_t \rarrow \hat{G}(\bar{E})$ extends $\tau$, and that $G/\Oc_F$ is a smooth group-scheme extending $G/k$.
  Write $K(1) = \ker(G(\Oc_F) \rarrow G(k))$.  Let $\Pi(\rho)$ be the representation of $G(F)$ associated to $\rho$ by
  the local Langlands correspondence\footnote{Precisely, $\rho \mapsto \Pi(\rho)$ is the inverse of the map $\rec_l$ in
    \cite{MR1876802} Section~VII.2.}, and assume that $\Pi(\rho)$ is \emph{generic}.  Then one can show that, as
  $G(k)$-representations,
  \[\Pi(\rho)^{K(1)} = \pi_G(\tau).\]
\end{lemma}
\begin{proof}
  We immediately reduce to the case $G = GL_n$.  If $\hat{L}$ and $L$ are as in Definition~\ref{def:piGtau}, and
  $L/\Oc_F$ is a Levi subgroup of $G/\Oc_F$ extending $L/k$, then for any $\rho$ as in the lemma we can conjugate $\rho$
  to have image in $\hat{L}(\bar{E})$.  We then have
  \[ \Pi(\rho) = \Ind_{P(F)}^{G(F)} \Pi_L(\rho)\]
  where $\Pi_L$ is the local Langlands correspondence for $L$ and $P$ is a parabolic subgroup with Levi $L$.  Taking
  $K(1)$-invariants we see that it suffices to prove the lemma in the case that $\tau$ is discrete.

  Let $M/\Oc_F$ be a split Levi subgroup, with dual $\hat{M}$, such that the semisimple part of $\tau$, $\tau_s$,
  factors through a discrete parameter $s : I_t \rarrow \hat{M}(E)$.  Then there is $w_0 \in W_M\subset W$ such that
  $w_0s = s^q$, and associated to the pair $(w_0, s)$ we have a representation $\epsilon(w_0)R_M(w_0, s)$ of $M(k)$
  which will be \emph{cuspidal} by \cite{MR0393266} Theorem~8.3.  We claim that $\pi_G(s)$ is the (unique) nondegenerate
  irreducible representation of $G(k)$ with cuspidal support given by the pair $(M(k), \epsilon(w_0)R(w_0, s))$.  Since
  $\pi_G(s)$ is nondegenerate by Theorem~\ref{thm:gelfand-mult}, it suffices to show that it has the given cuspidal
  support.  If $M \subset P$ is a parabolic subgroup defined over $k$, then
  \[\Ind_{M(k)}^{G(k)} R_M(w_0, s) = R(w_0, s)\] by \cite{MR0393266} Proposition~8.2, where $w_0$ is regarded as an
  element of both $W_M$ and $W$.  We have to show that
  \[\pres{\pi_G(s), \epsilon(w_0)R(w_0, s)} \neq 0.\] But, by \cite{MR0393266} Theorem~6.8, we have
  \begin{align*}\pres{\pi_G(s), \epsilon(w_0)R(w_0, s)}
    &= \frac{\epsilon(w_0)}{|W(s)|}\sum_{w \in W(s,s^q)} \epsilon(w) \pres{R(w,s), R(w_0, s)}\\
    &= \frac{\epsilon(w_0)}{|W(s)|} \sum_{w \in W(s, s^q)}\epsilon(w) |\{x \in W(s) : xwx^{-1} = w_0\}| \\
    &= \frac{\epsilon(w_0)}{|W(s)|} \sum_{x \in W(s)} \epsilon(xw_0x^{-1}) \\
    &= 1
  \end{align*}
  as required.  Now, the semisimplification of $\rho$ has the form $\rho_M$ for some $\rho_M : W_t \rarrow \hat{M}(F)$
  with $\rho_M|_{I_t} = s$.  Then $\Pi(\rho)$ will be a discrete series representation with supercuspidal support
  $(M, \nu)$ for some supercuspidal representation $\nu = \Pi_M(\rho_M)$.  It follows from \cite{shotton-gln}~Corollary
  6.21 parts~(1) and~(2) that $\Pi(\rho)^{K(1)}$ is the unique nondegenerate irreducible representation of $G(k)$ with
  cuspidal support $(M(k), \nu^{K(1) \cap M})$, and we have to show that $\nu^{K(1) \cap M} = \epsilon(w_0)R(w_0,s)$.
  Thus we have reduced to the cuspidal case, which boils down to comparing the construction of \cite{MR2480618} with the
  known local Langlands correspondence for general linear groups.  This is implicit in the remarks following Theorem~1.1
  of \cite{MR2676163}: we spell out the argument.

  We may suppose that $M = GL_n$ and $s : I_t \rarrow \hat{T}(E)$ is a discrete semisimple parameter.  Then
  \[s \cong \chi \oplus \chi^{\phi} \oplus \ldots \oplus \chi^{\phi^{n-1}}\] for some $\chi : I_t \rarrow \hat{T}(E)$,
  where $\chi^\phi$ is the twist of $\chi$ by $\phi \in W_t$, and $w_0 = (12\ldots n) \in W_M \cong S_n$.  Let $W'_t$ be
  the tame Weil group of the unramified extension $F_n/F$ of degree $n$.  Then $\chi$ extends to a character
  $\tilde{\chi}$ of $W'_t$ and $s = \left(\Ind_{W'_t}^{W_t}\tilde{\chi}\right)|_{I_t}$.  By \cite{MR1876802} Lemma~12.7
  part~(6), \[\Pi\left(\Ind_{W'_t}^{W_t}\tilde{\chi}\right) = \Ind_{F_n}^F(\Pi(\tilde{\chi})).\] Here $\Ind_{F_n}^F$
  denotes the cyclic automorphic induction of \cite{MR1328755}, which in this case agrees with the construction of
  \cite{MR1235293}.  We have that $\Pi(\tilde{\chi})|_{\Oc_{F_n}}^\times$ is inflated from the character $\theta$ of
  $k_n^\times$ corresponding to $\chi$ via the canonical surjection $I_t \onto k_n^\times$.  If we take $T \subset M$ to
  be a maximal torus of type $w_0$, then there is an isomorphism $T(k) \cong k_n^\times$.  It follows from the main
  theorem and Paragraph~3.4 of \cite{MR1235293} that $\left(\Ind_{F_n}^F(\Pi(\tilde{\chi}))\right)^{K(1)}$ is, as a
  representation of $K/K(1) = G(k)$, precisely $(-1)^{n-1}R^\theta_T = \epsilon(w_0)R(w_0, s)$, as required.
\end{proof}

\subsection{Endomorphisms of Gelfand--Graev representations}
\label{sec:endomorphisms}

Notice that the $q$-power stable $W$-orbits of $\Hom(I_t, \hat{G}(E))$ are exactly the $E$-points of the affine scheme
$\Sc^{\hat{G}}(q)$ introduced previously.  We write $\BqG$ for its ring of functions.
\begin{corollary} There are canonical isomorphisms
  \[\End_{G(k)}(\Gamma_{G,E}) \cong \prod_{[s]} E \cong \BqG\otimes E\]
  where $[s]$ runs over the $q$-power stable $W$-orbits of $\Hom(I_t, \hat{G}(E))$.
\end{corollary}

\begin{proof}
  The first isomorphism is the product of the ``Curtis homomorphisms''
  \[\End_{G(k)}(\Gamma_{G,E}) \rarrow \End_E(\pi_{G}(s)) = E.\]  The second takes the copy of $E$ labelled by $[s]$ to the
  copy of $E$ corresponding to the point $s$ of $\Sc_{\hat{G}}(q)$.
\end{proof}

\begin{remark} The problem of determining the \emph{integral} endomorphism ring $\End_{G(k)}(\Gamma_G)$ (for general
  connected reductive groups $G$) was considered by Bonnaf\'{e}--Kessar \cite{MR2441999}, who obtained a description
  when $l \nmid |W|$.  In the case $G = GL_n$, it is in fact true that the map
  $\BqG \rarrow \End_{G(k)}(\Gamma_G) \otimes E$ that we have obtained is an isomorphism of $\BqG$ onto
  $\End_{G(k)}(\Gamma_G)$.  This is proved in \cite{helm2016curtis} and \cite{1610.03277} as a byproduct of their proof
  of the local Langlands correspondence in families.\footnote{See the introduction for further remarks on this.}
\end{remark}

\begin{proposition}\label{prop:curtis} Let $L \subset G$ be a Levi subgroup.  Regard $\Ind_{L(k)}^{G(k)}(\Gamma_{L,E})$ as a module over
  $\BqL$ via the homomorphism
  \[\BqL \rarrow \BqL \otimes E \isomto \End(\Gamma_{L,E}) \rarrow \End\left(\Ind_{L(k)}^{G(k)}(\Gamma_{L,E})\right).\] Then, for
  each $[s] \in \Sc^{\hat{L}}(q)(E)$, we have an isomorphism of $G(k)$-representations
  \[\Ind_{L(k)}^{G(k)}(\Gamma_{L,E}) \otimes_{B_{q, \hat{L}}, [s]} E \cong \pi_G(\tau)\]
  where $\tau : I_t \rarrow \hat{L}(E)$ is a discrete inertial parameter with semisimple part $s$.
  \end{proposition}
\begin{proof}
  By the definition of $\pi_G(\tau)$, this immediately reduces to the case $L = G$, in which case it follows from the
  definition of the isomorphism $\BqG \rarrow \End_{G(k)}(\Gamma_{L,E})$ via Curtis homomorphisms.
\end{proof}

\subsection{Blocks and localisation}
\label{sec:blocks}

Let $\bar{s}$ be an $\FF$-point of $\Sc^{\hat{G}}(q)$, that is, a $q$-power stable semisimple conjugacy class in
$\Hom(I_t, \hat{G}(\FF))$.  Then \cite[Theorem~2.2]{MR983059} implies that the set of isomorphism classes of irreducible
representations that occur in some $R(w,s)$ is a union of blocks for $\Oc[G(k)]$. In particular, there is a central
idempotent $e_{\bar{s}} \in \Oc[G(k)]$ which acts as the identity precisely on these irreducible representations (and as
zero on the others).

Let $B_{q,\Ghat, \bar{s}}$ be the localisation of $\BqG$ at $\bar{s}$, and consider the projective $\Oc[G(k)]$-module
$e_{\bar{s}}\Gamma_G$ (a direct summand of $\Gamma_G$).  Then, again via the product of Curtis homomorphisms, we have a
homomorphism
\[B_{q, \Ghat, \bar{s}} \rarrow \End(e_{\bar{s}}\Gamma_{G,E}).\] Similarly, if $L \subset G$ is a Levi subgroup we have
a map
\[B_{q, \Lhat, \bar{s}} \rarrow \End(\Ind_{L(k)}^{G(k)}e_{\bar{s}}\Gamma_{L,E})\] and we obtain a corresponding version
of Proposition~\ref{prop:curtis}.

\section{The Breuil--M\'{e}zard conjecture}
\label{sec:combining}

If $X$ is any finite-dimensional scheme, let $\Zc(X)$ be the free abelian group on the irreducible components of $X$ of
maximal dimension.  If $X = \Spf A$ for $A \in \Cc_\Oc^\wedge$, then we write $\Zc(X) = \Zc(\Spec(A))$.

Let $G$ and $\hat{G}$ be as in Section~\ref{sec:rep-gln}, and suppose that $(E,\Oc, \FF)$ is sufficiently large in the
sense of assumption~(\ref{eq:sufflarge}).  Define a map \[\cyc : K_E(G(k)) \rarrow \Zc(\Xf^{\Ghat}(q))\] as follows: for
each isomorphism class of inertial $\hat{G}$-parameter $\tau : I_t \rarrow \hat{G}(E)$, there is an irreducible --- in
fact, geometrically irreducible --- component $\Cc_{\tau}$ of $\Xf^{\hat{G}}(q)$ such that $\rho_x|I_t\cong \tau$ for a
Zariski dense (open) set of $x \in \Cc_{\tau}(\bar{E})$.  Then for $\sigma$ an irreducible $E$-representation of $G(k)$
we define
\[ \cyc(\sigma) = \sum_{\tau} m(\sigma,\tau)[\Cc_\tau],\] where
$m(\sigma, \tau) = \dim \Hom_{G(k)}(\pi_G(\tau),\sigma)$, and we extend this linearly to $K_E(G(k))$.

\begin{remark} It follows from Lemma~\ref{lem:comparison} that $\cyc(\sigma) = \cyc'(\sigma^*)$ where $\cyc'$ is the
  cycle map defined in \cite{shotton-gln}~4.2 and $\sigma^*$ is the dual of $\sigma$.  The dual makes no difference to
  the following result.
\end{remark}

There are reduction maps $\red : K_E(G(k)) \rarrow K_{\FF}(G(k))$ and $\red : \Zc(\Xf^{\Ghat}(q)) \rarrow
\Zc(\Xf^{\Ghat}(q)_{\FF})$, the first defined by `choose a lattice, apply $\otimes_\Oc \FF$, and take the image in the
Grothendieck group' and the second defined by intersection with the special fibre, as in \cite{shotton-gln} Section~2.3.

\begin{theorem}\label{thm:tame-BM} There exists a homomorphism $\bar{\cyc} : K_{\FF}(G(k)) \rarrow \Zc(\Xf^{\Ghat}(q)_{\FF})$ such that the
  diagram
  \begin{equation}
\begin{CD}
  \label{eq:BM}
K_E(G(k)) @>{\cyc}>> \Zc(\Xf^{\Ghat}(q)) \\
@V{\red}VV  @V{\red}VV \\
K_{\FF}(G(k)) @>{\bar{\cyc}}>> \Zc(\Xf^{\Ghat}(q)_{\FF}).
\end{CD}
\end{equation}
\end{theorem}

\begin{proof}
  Let $f$ be an integer large enough for $\hat{G}$ (see Definition~\ref{def:large}).  By Lemma~2.10 of
  \cite{shotton-gln}, it is enough prove the theorem after enlarging $\Oc$. Then, by \cite[Proposition~7.1]{shotton-gln}
  and Lemma~\ref{lem:unique-cpt}, it suffices to prove the theorem with $\Xf^{\Ghat}(q)$ replaced by
  $X_{\rhobar}^{\hat{G}}$ for $\rhobar$ an $f$-distinguished $\FF$-point of $\Xf^{\Ghat}(q)$.  Let $\rhobar$ be such a
  point and let $\hat{L}$ be an allowable Levi subgroup for $\rhobar$.  By Theorem~\ref{thm:def-rings}, there is a
  formally smooth morphism
  \[X_{\rhobar}^{\hat{G}} \rarrow S_{\bar{s}}^{\hat{L}}.\]
  We have that $S_{\bar{s}}^{\hat{L}} = \Spec B_{q, \Lhat, \bar{s}}$ and that $B_{q, \Lhat, \bar{s}}$ is a finite flat local
  $\Oc$-algebra.

  It follows from this that $\Zc(X_{\rhobar}^{\hat{G}} \otimes \FF) \cong \ZZ$ is generated by the class of the unique
  irreducible component, and $\Zc(\Xf^{\Ghat})$ is the free abelian group on the $E$-points $[s]$ of
  $S_{\bar{s}}^{\hat{L}}$.  With these identifications, by Theorem~\ref{thm:def-rings} the reduction map on the right is simply
  \[\sum a_{[s]} [s] \mapsto \sum a_s,\]
  and we seek a map $\bar{\cyc} : K_{\FF}(G(k)) \rarrow \ZZ$ such that
  \[\bar{\cyc}(\bar{\sigma}) = \sum_{[s]} m(\sigma,\tau_{[s]})\]
  for all $\sigma \in K_E(G(k))$.

  Let $\Theta = \Ind_{L(k)}^{G(k)} e_{\bar{s}}\Gamma_L$.  Then $\Theta$ is a finitely generated projective
  $\Oc[G(k)]$-module by Lemma~\ref{lem:projective}, the fact that $e_{\bar{s}}$ is an idempotent, and the fact that
  $\Ind$ takes projectives to projectives.  If $\Theta_E = \Theta \otimes E$ then we have a homomomorphism
  \[B_{q,\hat{L},\bar{s}} \rarrow \End_{G(k)}(\Theta_E)\]
  from Sections~\ref{sec:endomorphisms} and \ref{sec:blocks}.  For any $\Oc[G(k)]$-representation $\sigma$, define
  $\Theta(\sigma)$ to be $\Hom_{\Oc[G(k)]}(\Theta, \sigma)$, an exact functor of $\sigma$.  I claim that $\bar{\cyc}$
  can be defined by setting
  \[ \bar{\cyc}(\nu) = \dim_\FF \Theta(\nu)\]
  for irreducible representations $\nu$ of $G(k)$ over $\FF$, and extending linearly.  Note that, since $\Theta(\cdot)$
  is exact, if $\omega$ is any representation of $G(k)$ over $\FF$ with image $[\omega]$ in $K_\FF(G(k))$, then
  \[ \bar{\cyc}([\omega]) = \dim_\FF \Theta(\omega).\]

  Indeed, for $\sigma$ an irreducible $E$-representation of $G(k)$ admitting a lattice $\sigma^{\circ}$ we have
  \begin{align*}
    \sum_{[s]} m(\sigma, \tau_{[s]}) &= \sum_{[s]} \dim \Hom_{E[G(k)]}(\Theta_E \otimes_{B_{q,\hat{G}, \bar{s}},
                                       [s]} E, \sigma)  \\
    \intertext{(by Proposition~\ref{prop:curtis} and the discussion of section~\ref{sec:blocks})}
                                     &= \dim_E \Theta(\sigma) \\
                                     &= \rank_\Oc \Theta(\sigma^\circ)\\
                                     &= \dim_\FF\Theta(\bar{\sigma})\\
    \intertext{(by projectivity of $\Theta$)}
                                     &= \bar{\cyc}(\bar{\sigma})
  \end{align*}
  as required.  The theorem follows.
\end{proof}
\bibliography{references.bib}{}

\newcommand{\noop}[1]{}
\providecommand{\bysame}{\leavevmode\hbox to3em{\hrulefill}\thinspace}
\providecommand{\MR}{\relax\ifhmode\unskip\space\fi MR }
\providecommand{\MRhref}[2]{%
  \href{http://www.ams.org/mathscinet-getitem?mr=#1}{#2}
}
\providecommand{\href}[2]{#2}
\begin{thebibliography}{{Sta}17}

\bibitem[BK08]{MR2441999}
C\'{e}dric Bonnaf\'{e} and Radha Kessar, \emph{On the endomorphism algebras of
  modular {G}elfand-{G}raev representations}, J. Algebra \textbf{320} (2008),
  no.~7, 2847--2870. \MR{2441999}

\bibitem[BM89]{MR983059}
Michel Brou\'{e} and Jean Michel, \emph{Blocs et s\'{e}ries de {L}usztig dans
  un groupe r\'{e}ductif fini}, J. Reine Angew. Math. \textbf{395} (1989),
  56--67. \MR{983059}

\bibitem[Cho17]{MR3622123}
Suh~Hyun Choi, \emph{Local universal lifting spaces of {${\rm mod}\, l$}
  {G}alois representations}, J. Number Theory \textbf{176} (2017), 113--148.
  \MR{3622123}

\bibitem[CHT08]{ClozelHarrisTaylor2008-Automorphy}
Laurent Clozel, Michael Harris, and Richard Taylor, \emph{Automorphy for some
  $l$-adic lifts of automorphic mod $l$ {{G}alois} representations}, Publ.
  Math. Inst. Hautes Études Sci. (2008), no.~108, 1--181, With Appendix A,
  summarizing unpublished work of Russ Mann, and Appendix B by Marie-France
  Vignéras.

\bibitem[Con12]{conrad2012weil}
Brian Conrad, \emph{Weil and {G}rothendieck approaches to adelic points},
  L’Enseignement Math{\'e}matique \textbf{58} (2012), no.~1, 61--97.

\bibitem[DG67]{EGA}
Jean Dieudonn{\'e} and Alexander Grothendieck, \emph{\'{E}l\'ements de
  g\'eom\'etrie alg\'ebrique}, Inst. Hautes \'Etudes Sci. Publ. Math.
  \textbf{4, 8, 11, 17, 20, 24, 28, 32} (1961--1967).

\bibitem[DL76]{MR0393266}
P.~Deligne and G.~Lusztig, \emph{Representations of reductive groups over
  finite fields}, Ann. of Math. (2) \textbf{103} (1976), no.~1, 103--161.

\bibitem[DM91]{MR1118841}
Fran\c{c}ois Digne and Jean Michel, \emph{Representations of finite groups of
  {L}ie type}, London Mathematical Society Student Texts, vol.~21, Cambridge
  University Press, Cambridge, 1991. \MR{1118841}

\bibitem[DR09]{MR2480618}
Stephen DeBacker and Mark Reeder, \emph{Depth-zero supercuspidal {$L$}-packets
  and their stability}, Ann. of Math. (2) \textbf{169} (2009), no.~3, 795--901.
  \MR{2480618}

\bibitem[EH14]{1104.0321}
Matthew Emerton and David Helm, \emph{The local {L}anglands correspondence for
  {${\rm GL}_n$} in families}, Ann. Sci. \'{E}c. Norm. Sup\'{e}r. (4)
  \textbf{47} (2014), no.~4, 655--722. \MR{3250061}

\bibitem[Hel]{helm2016curtis}
David Helm, \emph{Curtis homomorphisms and the integral {B}ernstein center for
  {$GL_n$}}, preprint available at \url{https://arxiv.org/abs/1605.00487}.

\bibitem[Hen92]{MR1235293}
Guy Henniart, \emph{Correspondance de {L}anglands-{K}azhdan explicite dans le
  cas non ramifi\'{e}}, Math. Nachr. \textbf{158} (1992), 7--26. \MR{1235293}

\bibitem[HH95]{MR1328755}
Guy Henniart and Rebecca Herb, \emph{Automorphic induction for {${\rm GL}(n)$}
  (over local non-{A}rchimedean fields)}, Duke Math. J. \textbf{78} (1995),
  no.~1, 131--192. \MR{1328755}

\bibitem[HM18]{1610.03277}
David Helm and Gilbert Moss, \emph{Converse theorems and the local {L}anglands
  correspondence in families}, Invent. Math. \textbf{214} (2018), no.~2,
  999--1022. \MR{3867634}

\bibitem[HT01]{MR1876802}
Michael Harris and Richard Taylor, \emph{The geometry and cohomology of some
  simple {S}himura varieties}, Annals of Mathematics Studies, vol. 151,
  Princeton University Press, Princeton, NJ, 2001, With an appendix by Vladimir
  G. Berkovich.

\bibitem[Kis09]{Kisin2009-FontaineMazur}
Mark Kisin, \emph{The {Fontaine}--{Mazur} conjecture for $\mathbf{{GL}}_2$}, J.
  Amer. Math. Soc. \textbf{22} (2009), no.~3, 641--690.

\bibitem[Pa{\v{s}}15]{MR3306557}
Vytautas Pa{\v{s}}k{\=u}nas, \emph{On the {B}reuil-{M}\'ezard conjecture}, Duke
  Math. J. \textbf{164} (2015), no.~2, 297--359. \MR{3306557}

\bibitem[SF99]{stanley1999enumerative}
R.P. Stanley and S.~Fomin, \emph{Enumerative combinatorics:}, Cambridge Studies
  in Advanced Mathematics, no. v. 2, Cambridge University Press, 1999.

\bibitem[Sho18]{shotton-gln}
Jack Shotton, \emph{The {B}reuil-{M}\'{e}zard conjecture when {$l\neq p$}},
  Duke Math. J. \textbf{167} (2018), no.~4, 603--678. \MR{3769675}

\bibitem[{Sta}17]{stacks-project}
The {Stacks Project Authors}, \emph{\itshape {Stacks Project}},
  \url{http://stacks.math.columbia.edu}, 2017.

\bibitem[Yos10]{MR2676163}
Teruyoshi Yoshida, \emph{On non-abelian {L}ubin-{T}ate theory via vanishing
  cycles}, Algebraic and arithmetic structures of moduli spaces ({S}apporo
  2007), Adv. Stud. Pure Math., vol.~58, Math. Soc. Japan, Tokyo, 2010,
  pp.~361--402. \MR{2676163}

\end{thebibliography}
\bibliographystyle{amsalpha}
\end{document}